\documentclass[12pt]{amsart}
\usepackage{amsmath,amsfonts,amsthm,amssymb,amscd,stmaryrd,url}
\DeclareFontFamily{OT1}{rsfs}{}
\DeclareFontShape{OT1}{rsfs}{n}{it}{<-> rsfs10}{}
\DeclareMathAlphabet{\mathscr}{OT1}{rsfs}{n}{it}
\title[A discrete mean value]{A discrete mean value of the 
derivative of the Riemann zeta function}
\author{Nathan Ng}
\newcommand{\myref}[1]{(\ref{#1})}
\newtheorem{thm}{Theorem}[section]

\newtheorem{lem}[thm]{Lemma}
\newtheorem{prop}[thm]{Proposition}

\newcommand{\Log}{ \log(\mbox{$\frac{T}{2 \pi}$})}
\newcommand{\Lom}{\mbox{$\frac{\log M}{\log \log M}$}}

\newcommand{\xinf}{{||{\bf x}||_{\infty}}}
\newcommand{\yone}{{||{\bf y}||_1}}

\newcommand{\chiq}{ \sideset{}{^*}\sum_{\psi \, \mathrm{mod} \, q}}
\newcommand{\Lo}{{\mathscr{L}}}

\begin{document}
\maketitle
{\def\thefootnote{}
\footnote{\today. \ {\it Mathematics Subject Classification (2000)}.
11M26. \\
This research was funded in part by NSERC and NSF FRG grant DMS 0244660. 
}}
\begin{abstract}
\noindent
In this article we compute a discrete mean value of the derivative of the
Riemann zeta function.  This mean value will be important for several 
applications concerning the size of $\zeta'(\rho)$ where $\zeta(s)$ 
is the Riemann zeta function and $\rho$ is a non-trivial zero of the 
Riemann zeta function. 
\end{abstract}
\section{Introduction}

In this article we compute a discrete mean value of the Riemann
zeta function, $\zeta(s)$.  Throughout, $\rho=\beta+i\gamma$ 
will denote a non-trivial zero of the Riemann zeta function and
$T$ will be a large parameter.  Moreover, 
we define the Dirichlet polynomials
\[
    X(s) = \sum_{n \le M} \frac{x_n}{n^{s}} \ \mathrm{and} \
    Y(s) = \sum_{n \le M} \frac{y_n}{n^{s}}
\]
where $x_n$ and $y_n$ are arbitrary real sequences and $M = T^{\theta}$
with $0 < \theta < 1/2$.  We shall evaluate 
the following mean value:
\begin{equation}
  S = \sum_{0 < \gamma < T} \zeta'(\rho) X(\rho) Y(1-\rho)   \ . 
\end{equation} 
However, our main purpose for evaluating $S$
is to employ it for an application concerning  
large values of $\zeta'(\rho)$.
In an accompanying paper \cite{Ng3} we prove the following results:
\begin{thm} \label{thm3}
Assume the Riemann Hypothesis. 
For each $A > 0$,  we have 
\[
       |\zeta'(\rho)| \gg (\log |\gamma|)^{A} 
\]
for infinitely many $\rho= \frac{1}{2}+i\gamma$. 
\end{thm}
In order to strengthen this result we will require an additional assumption 
concerning the location of zeros of Dirichlet $L$-functions.  \\

\noindent {\bf Large zero-free region conjecture.}  \\
There exists a constant $c_0>0$ sufficiently large such that
for each $q \ge 1$ and each character $\chi$ modulo $q$ the Dirichlet $L$-function
$L(s,\chi)$ does not vanish in the region 
\[
     \sigma \ge 1 - \frac{c_0}{\log \log (q (|t|+4))}
\]
where $s= \sigma + it$.  \\

The value of $c_0$ required may be calculated and $c_0=100$ suffices but is not minimal.
We note that this large zero-free region conjecture is significantly weaker than the Generalized Riemann Hypothesis which asserts that all non-trivial zeros of each Dirichlet $L$-function lie on the one-half line.  On the other hand, this 
zero-free region is still much larger than what has currently been proven. 
For example, this conjecture rules out the existence of Siegel zeros. \\

\begin{thm} \label{thm4}
Assume the Riemann hypothesis and the large zero-free conjecture for 
Dirichlet $L$-functions.   There exists a constant $c_2 >0$ such that
\[
       |\zeta'(\rho)| \gg  \exp 
       \left(
       c_2 \sqrt{\frac{\log|\gamma|}{\log \log |\gamma|}}
       \right) 
\]
for infinitely many $\rho = \frac{1}{2} + i\gamma$. 
\end{thm}
 The methods employed to prove Theorems \ref{thm3} and 
\ref{thm4} are based on Soundararajan's \cite{So} resonance method.  This 
method has proven to be successful in determining extreme values of $L$-functions and character sums.  
In the article \cite{Ng3} we also succeed in exhibiting small values of 
$|\zeta'(\rho)|$ too. \\

We also note that an asymptotic evaluation of $S$ has other important applications.  Soundararajan has informed me that he can prove under the
assumption of the Riemann hypothesis that 
\[
    \sum_{0 < \gamma < T} |\zeta'(\rho)|^{2k} 
    \gg_{k} T (\log  T)^{(k+1)^2} \ . 
\]
This proof requires our formula for $S$ and follows the lower bound
method of Rudnick and Soundararajan \cite{RS1}, \cite{RS2}. We observe
that this is stronger than Theorem \ref{thm3}. \\

Our evaluation of $S$ will be split in two cases depending on the properties
of the coefficients $x_n$ and $y_n$.  The cases we shall consider are:  \\

\noindent {\bf Case 1.} {\it  The divisor case.} The coefficients shall satisfy the
bounds
\begin{equation}
   |x_n|, |y_n| \le \tau_{r}(n) (\log T)^{C} 
   \label{eq:divc}
\end{equation}
for $r \in \mathbb{N}$ and $C \ge 0$
where $\tau_{r}(n)$ is the $r$-th divisor function.
We shall also assume
\begin{equation}
   |x_{mn} | \ll |x_m| |x_n| \ \mathrm{and} \
   |y_{mn}| \ll |y_m| |y_n|  \ . 
   \label{eq:divc2}
\end{equation}
\noindent {\bf Case 2.} {\it The resonator case.}  In this case we will
take $x_n=y_n=f(n)$ where $f$ is a multiplicative
function supported on the squarefree integers.  For a prime $p$ we 
define 
\begin{equation}
    f(p) = \left\{ \begin{array}{ll}
                  \frac{L}{\log p} & \mbox{if $L^2 \le p \le 
                  \exp((\log L)^2)$} \\
                  0 & \mbox{else} \\
                  \end{array} \right. 
                  \label{eq:resc}
\end{equation}
where $L = \sqrt{\log M \log \log M}$.  \\

The resonator coefficients have recently been employed by Soundararajan \cite{So}
and they arose in a certain optimization problem related to finding extreme values of $\zeta(1/2+it)$. 
We shall refer throughout this article to case one as the divisor case and 
case two as the resonator case.  Our evaluation of $S$ in the divisor case will be
unconditional whereas the evaluation of $S$ in the resonator case will depend
on the (as yet) unproven large zero-free region conjecture.  \\

We now state our result for $S$. 
We let $c_j$ for $j=1,2,3, \ldots$ denote positive constants. 
\begin{thm} \label{thm1}
$(i)$  If $x_n,y_n$ satisfy~(\ref{eq:divc}),~(\ref{eq:divc2}) then
\begin{equation}
\begin{split}
   S & = \frac{T}{2 \pi}
  \sum_{nu \le M} \frac{x_{u}y_{nu}r_{0}(n)}{nu}
   -   \frac{T}{4 \pi} \sum_{nu \le M} \frac{y_{u}x_{un}}{nu} 
    R_2 \left( \log(\mbox{$\frac{T}{2 \pi n}$}) \right)  \\
  & +\frac{T}{2 \pi} \sum_{{\begin{substack}{a,b \le M
         \\ (a,b)=1}\end{substack}}}
   \frac{r_1(a,b)}{ab}
   \sum_{g \le \min(\frac{M}{a},\frac{M}{b})} \frac{y_{ag}x_{bg}}{g}
   + \tilde{\mathcal{E}}
   \label{eq:S1ass}
\end{split}
\end{equation}
where for $0 < \theta < 1/2$ we have for any $A' > 0$
\[
     \tilde{ \mathcal{E}} \ll_{A'} T (\log T)^{-A'} +
   T^{\frac{3}{4}+\frac{\theta}{2}+\epsilon}  \ . 
\]
The other quantities are defined as follows:
\[
   r_0 (n) =  P_2(\Log)-
    2P_1(\Log)(\log n)+  (\Lambda*\log)(n)  \ , 
\]
\[
    r_1(a,b) = \mbox{$\frac{1}{2}$} \Lambda_2(a)
   - R_1 \left(  \log(\mbox{$\frac{T}{2 \pi b}$})   \right) \Lambda(a)
   - \tilde{R}_1 \left( 
   \log(\mbox{$\frac{T}{2 \pi b}$})
   \right)  \alpha_1(a) - \alpha_2(a)  \ , 
\]
$P_1,P_2,R_1,\tilde{R}_1,R_2$
are monic polynomials of degrees $1,2,1,1,2$ respectively and $\alpha_1,\alpha_2$ 
are arithmetic functions.  In fact, $\alpha_1$ is supported on prime powers,
$\alpha_2$ is supported on integers $n$ such that $\omega(n) \le 2$.
More precisely,   $\alpha_1(p^{\alpha}) = \frac{\log p}{p-1}$, 
   $ \alpha_2(p^{\alpha})   =
  \mbox{$  - \frac{(\alpha+1) (\log p)^2}{p-1}
    + \frac{D\log p}{p-1} - \frac{\log p}{(p-1)^2}$}$
    for some $D \in \mathbb{R}$, and 
 $\alpha_2 (p^{\alpha} q^{\beta}) = \mbox{$- (\log p)(\log q) (
   \frac{1}{p-1} + \frac{1}{q-1})$}$ for $\alpha,\beta \in \mathbb{N}$.   \\

\noindent  $(ii)$ Assume the large zero-free region conjecture. 
If $x_n=y_n=f(n)$ where $f$ is defined by~(\ref{eq:resc}) then
there exists a $c_1 > 0$ such that~(\ref{eq:S1ass})
remains true with  
\[
    \tilde{\mathcal{E}} \ll T \exp \left(
    -  \mbox{$\frac{c_1 \log T}{\log \log T}$}
    \right) + 
     T^{\frac{3}{4}+\frac{3\theta}{2}+\epsilon} 
\] 
for $0 < \theta < 1/6$. 
\end{thm}

\noindent {\bf Remarks. 1.}  It is possible to obtain an intermediate result
to Theorems \ref{thm3} and \ref{thm4}.  In fact, one can show that if
the Riemann hypothesis is true and there are no Siegel zeros then  there
exists a $c_3>0$ such that
\[
   |\zeta'(\rho)| \gg \exp (c_3 (\log |\gamma|)^{1/4})
\]
infinitely often.   The proof of this result rests on deriving Theorem \ref{thm1} 
for the sequences $x_n=y_n=g(n)$ where $g$ is multiplicative and
supported on squarefree integers.  On primes it is defined by 
\[
         g(p) = \left\{ \begin{array}{ll}
                  \frac{c(\log M)^{1/4}\log \log M}{\log p} & \mbox{if $A \le p \le 
                  B$} \\
                  0 & \mbox{else} \\
                  \end{array} \right. 
\]
for some $c>0$. The evaluation of $S$ in this case is very similar to the two other cases worked
out in the article.  
 However, due to the length of this article we have decided not to present this case. \\

\noindent {\bf 2.}   Various mean values involving 
$\zeta'(\rho)$ have been explored in several previous articles. 
(See \cite{CGG2}, \cite{CGG3}, \cite{Fu}, \cite{Go}, \cite{Ng0}.)  
 Discrete moments of $\zeta'(\rho)$ have number theoretic applications to 
simple zeros of the zeta function \cite{CGG3} and to the 
distribution of the summatory function of the M\"{o}bius function
\cite{Ng1}.  More generally, moments of $X(\rho+\alpha)$ for a 
Dirichlet polynomial $X(s)$ and $\alpha \in \mathbb{C}$ have applications
to extreme gaps between the zeros of the zeta function (see \cite{CGG1},\cite{Ng2}). \\

\noindent {\bf 3.} The argument for this theorem is based 
on an argument of Conrey, Ghosh, and Gonek \cite{CGG3} for
evaluating the sums
\[
  S_1 = \sum_{0 < \gamma < T} \zeta'(\rho)X(\rho) \ \mathrm{and} \
  S_2 = \sum_{0 < \gamma < T} \zeta'(\rho)\zeta'(1-\rho)X(\rho)
  X(1-\rho)
\]
where $x_n = \mu(n) P \left( \frac{\log(M/n)}{\log M} \right)$
and $P$ is a polynomial.  The evaluation of $S_2$ 
requires the assumption of the generalized Lindel\"{o}f hypothesis
whereas $S_1$ may be computed unconditionally. 
Oddly, this point is never stressed in \cite{CGG3}.  In a future article, we shall evaluate the above sums with  arbitrary coefficients by the methods of this article. \\ 

\noindent {\bf 4.}  The proof of the bound for $\tilde{\mathcal{E}}$ in 
Theorem 1 is obtained by an argument which is very similar to the
proof of the Bombieri-Vinogradov theorem.  Recall that it asserts that for each $A>0$
\[
    \sum_{q \le M} 
    \max_{(a,q)=1} \max_{y \le T}
    \left|\psi(y;a,q)-\frac{y}{\phi(q)} \right|
    \ll T (\log T)^{-A} + T^{1/2}M(\log TM)^{6} \ . 
\]
Improving our value of $M$ beyond $\sqrt{T}$ in Theorem \ref{thm1} lies 
as deep as improving the Bombieri-Vinogradov theorem for $M$ larger
than $\sqrt{T}$.  \\

\noindent{\bf 5}.  In this article we must impose some conditions on the location of 
zeros of Dirichlet $L$-functions.  More precisely, we assume the large zero-free region
conjecture for Dirichlet $L$-functions in order to evaluate $S$ in the resonator case.  The central reason for applying this conjecture is that the coefficients
$f(n)$ satisfy $\sum_{n \le M} f(n)^2 \sim M \exp(\frac{c \log T}{\log \log T})$.
Now the general setup for evaluating $S_1$ in the article of \cite{CGG3} is 
to use an argument similar to proving the Bombieri-Vinogradov theorem.  
However, this type of argument provides a savings of $(\log T)^{-A}$ for any
$A$ from the main term.  As the resonator coefficients $x_n =y_n=f(n)$ become very large in mean square we will be unable to obtain an asymptotic formula with only a savings of a power of a logarithm.  The   central reason for applying the large zero free region conjecture is that it allows us to have
a savings of $\exp(-\frac{c\log T}{\log \log T})$ for a large enough $c$
which balances the large average size of the resonator coefficients. 
On the other hand, for the evaluation of $S_2$ in \cite{CGG3} the coefficients $x_n,y_n$ are bounded in size.  In that case the Generalized
Lindel\"{o}f hypothesis is invoked in order to bound a sixth integral moment 
of $L(s,\chi)$ on average on the critical line and has nothing to do with 
the size of the coefficients $x_n,y_n$ as in our case.   \\

\noindent {\bf Acknowledgements}.  I would like to thank K. Soundararajan
for suggesting this problem.  Also thanks to H. Kadiri for a useful
question. 

\section{Notation}
Throughout out this article we shall denote a series of positive constants
by $c_j$ and $C_j$ for $j=1,2, \ldots$.  We remark that some of the constants
$C_j$ will depend on the numbers $r$ and $C$ given in~(\ref{eq:divc}).  
 For $T$ large we define 
$\Lo= \log(T) $. 
We shall also consider arbitrary sequences ${\bf x}=\{ x_n \}$  supported on the interval $[1,M]$. We shall employ 
the notation 
\[
    ||{\bf x} ||_{\infty} = \max_{n \le M} |x_n| \ \mathrm{and} \
    ||{\bf x}||_{p} = ( \sum_{n \le M} |x_n|^{p} )^{1/p} \ . 
\]

We shall use Vinogradov's notation $f(x) \ll g(x)$ to mean there exists a $C>0$
such that
$
    |f(x)| \le C g(x)
$
for all $x$ sufficiently large.  We denote $f(x)=O(g(x))$ to mean the same thing. 

In addition, we will encounter a host of familiar arithmetic functions. 
Let $\omega(n)$ denote the number of distinct prime factors of $n$. 
For $r>0$ we define $\tau_{r}(n)$, the $r$-th divisor function, to be the coefficient
of $n^{-s}$ in the Dirichlet series $\zeta^{r}(s)$.  If $r=2$ we 
write $\tau(n)=\tau_2(n)$. 
Similarly $\Lambda(n)$ is the coefficient of $n^{-s}$ in the Dirichlet series of $-\zeta'(s)/\zeta(s)$. This
yields the expression
$\Lambda(n) = \sum_{d \mid n} \mu(d) \log \frac{n}{d}$.
Moreover, we have its generalization $\Lambda_{k} = \mu *\log^{k}$. 
An equivalent definition is that
$\Lambda_{k}(n)$ is the coefficient of $n^{-s}$ in the Dirichlet series $(-1)^{k}\zeta^{(k)}(s)/\zeta(s)$. Furthermore,
$\Lambda_{k}(n)=0$ is supported on those integers with at most 
$k$ prime factors. 
We also define
$
   j(n) = \prod_{p \mid n} \left(1 + 10p^{-1/2} \right) 
$.

\section{The Dirichlet polynomial coefficients $x_n$} \label{coeff}

 We record some properties
of the coefficients that will be employed throughout the article:  \\

\noindent {\bf Properties of the divisor coefficients}. \\
Let $x_n$ and $y_n$ satisfy (\ref{eq:divc}) and (\ref{eq:divc2}). We  have the standard estimates:
\begin{equation}
\begin{split}
   & ||{\bf x}||_{\infty} \ , \ ||{\bf y}||_{\infty} \ll T^{\epsilon} \ ,  \\
& || \mbox{$\frac{x_n}{n}$}||_{1}  \ , \
     || \mbox{$\frac{y_n}{n}$}||_{1}  \ , \ 
      || \mbox{$\frac{x_n^2}{n}$}||_{1}  \ , \
    || \mbox{$ \frac{y_n^2}{n}$}||_{1}   \ , \
     || \mbox{$ \frac{x_n ( \tau_{k}*y)(n)}{n}$}|| \ll \Lo^{C'} 
    \label{eq:divbds}
\end{split}
\end{equation}
for some $C'>0$ and $k \in \mathbb{N}$.  We remark that the above
bounds remain true when the above sequences are multiplied by $j(n)$. \\

\noindent {\bf Properties of the resonator coefficients}. \\
Let $f$ be defined by (\ref{eq:resc}).  We have the following estimates:
\begin{equation}
\begin{split}
  f(mn) & \ll f(m) f(n) \ , \\
 ||f||_{\infty} & \le M^{1/2+\epsilon} \ \mathrm{for} \ M \
 \mathrm{sufficiently \ large}  \ ,  \label{eq:finf} \\
  ||f||_{1} & \ll  M\exp \left((1+o(1))
   \sqrt{\Lom} \right) \ , \\
\end{split}
\end{equation}
\begin{equation}
\begin{split}
   ||f^2||_1
   & \ll  M \exp \left((0.5+o(1)) \Lom \right)  \ ,  \\
    \sum_{n \le M} \frac{f(n)}{n} & \ll \exp
 \left((1+o(1)) \sqrt{\Lom} \right) \ ,    \\
  \sum_{n \le M} \frac{f(n)^2}{n} & \ll 
    \exp \left((0.5+o(1))\Lom \right)  \ , 
    \label{eq:f2n}
\end{split}
\end{equation}
 \begin{equation}
   \sum_{n \le M} \frac{j(n) (\tau_r*f)(n) f(n)}{n}
   \ll \exp \left((0.5+o(1))\Lom  \right)  \ , 
   \label{eq:resb}
\end{equation}
\begin{equation}
    \sum_{n \le T} \frac{(\tau*f)(n)^2}{n} \ , \
    \sum_{n \le T} \frac{(f*f)(n)^2}{n} 
    \ll T^{\epsilon} \ .
    \label{eq:resb2} 
\end{equation}

We now give an indication of how to prove \myref{eq:resb}. The proofs of the
other inequalities are similar.  We denote $\Sigma$ the sum to be estimated.
\begin{equation}
\begin{split}
   \Sigma & 
   \le \sum_{n=1}^{\infty} \frac{j(n)(\tau_r*f)(n)f(n)}{n}  
    = \prod_{p} \left( 1 +  \frac{j(p)(\tau_r*f)(p)f(p)}{p}
 \right) \ . 
 \nonumber
\end{split}
\end{equation}
Since $j(p)=1+O(p^{-1/2}), f(p)=\frac{L}{\log p},
 (\tau_r*f)(p)=\frac{L}{\log p}+r$ we obtain  
\[
   \log(\Sigma) \le \sum_{L^2 \le p \le \exp((\log L)^2)}
   \left( 
   \frac{L^2}{p(\log p)}
    + O \left( \frac{L}{p(\log p)} +\frac{L^2}{p^{3/2}} \right)
   \right) \ . 
\]
By the prime number theorem  
\[
   \sum_{L^2 \le p \le \exp((\log L)^2)} 
   \frac{1}{p (\log p)^2} = \frac{1}{8 (\log L)^2} (1+o(1))
\]
and thus
$
   \Sigma \le \exp \left(
   \frac{\log M}{2 \log \log M} (1+ o(1))
   \right)$. 

\section{Preliminary manipulations of $S$ and proof of Theorem \ref{thm1}}

\begin{proof} We commence with our evaluation of $S$.  We start with some 
initial manipulations.  Recall that our goal is to evaluate
\[
    S = \sum_{0 < \gamma < T} \zeta'(\rho) X(\rho) Y(1-\rho)  \ . 
\]   
The functional equation for the Riemann zeta function is
$
   \zeta(s) = \chi(s) \zeta(1-s) 
$, where
\[
   \chi(1-s) = \chi(s)^{-1} = 2(2 \pi)^{-s} \Gamma(s) \cos (\pi s/2) \ .
\]
Differentiating the functional equation,
\[
  \zeta'(s) = - \chi(s) \left( \zeta'(1-s) - \frac{\chi'}{\chi}(s) \zeta(1-s) \right) \ .
\]
From this last equation it follows that
\begin{align*}
   S 
   & = -\sum_{0 < \gamma < T} \chi(\rho) \zeta'(1-\rho) X(\rho) Y(1-\rho)  \\
   & = \frac{1}{2 \pi i} \int_{\mathcal{C}} \frac{\zeta'}{\zeta}(1-s) \chi(s) \zeta'(1-s) X(s)
   Y(1-s) \, ds
\end{align*}
where $\mathcal{C}$ is the positively oriented rectangle with vertices at $1-\kappa+i,\kappa+i,\kappa+iT,$ and $1-\kappa+iT$, and
$\kappa=1+\Lo^{-1}$.  Moreover, we choose $T$ so that the distance from $T$ to the nearest zero is $ \gg \Lo^{-1}$.   The bottom edge of this contour
is clearly $O(1)$.  On the top edge we have the standard bounds 
\begin{equation}
\begin{split}
   \chi(s) & \ll T^{1/2-\sigma}  \ , \\
   \zeta'(1-s) & \ll  T^{\sigma/2+\epsilon} \ , \\
   X(s) & \ll M^{1-\sigma} || \mbox{$ \frac{x_n}{n}$}||_{1}  \ , \\
   Y(1-s) & \ll M^{\sigma} || \mbox{$ \frac{y_n}{n}$}||_{1} \ , \\
   \frac{\zeta'}{\zeta}(1-s) & \ll \Lo^2  \ . 
   \nonumber
\end{split}
\end{equation}
Note that the last bound only holds for $s=\sigma+it$ as long as $|t-\gamma|
\gg \Lo^{-1}$ for all imaginary ordinates $\gamma$.  Combining these bounds
shows that the top edge of the contour is bounded by 
$
    MT^{1/2+\epsilon}
$.
Next note
that
\[
  \frac{\zeta'}{\zeta}(1-s) = \frac{\chi'}{\chi}(s) - \frac{\zeta'}{\zeta}(s)
\]
and
\[
  \chi(s) \zeta'(1-s) = - \zeta'(s) + \frac{\chi'}{\chi}(s) \zeta(s)
\]
imply the right-hand side of the integral is
\begin{equation}
  S_{R} = \frac{1}{2 \pi i} \int_{\kappa+i}^{\kappa+iT} \left(
  \frac{\chi'}{\chi}(s)^{2} \zeta(s) - 2 \frac{\chi'}{\chi}(s) \zeta'(s)
  + \frac{\zeta'}{\zeta}(s) \zeta'(s)
  \right) X(s) Y(1-s) \, ds \ .
  \label{eq:SR0}
\end{equation}
The left-hand side is
\[
  S_{L} = \frac{1}{2 \pi i} \int_{1-\kappa+iT}^{1-\kappa+i} \frac{\zeta'}{\zeta}(1-s) \chi(s) \zeta'(1-s)X(s)
   Y(1-s) \, ds \ .
\]
By the variable change $s \to 1-s$ the left side equals $-\overline{I}_{L}$ where
\[
   I_{L} = \frac{1}{2 \pi i} \int_{\kappa+i}^{\kappa+iT}  \chi(1-s) \frac{\zeta'}{\zeta}(s) \zeta'(s) X(s) Y(1-s)
   \, ds  \ . 
\]
We have now demonstrated that
\[
   S = S_{R}- \overline{I}_{L} + O(MT^{1/2+\epsilon}) 
\]
with $S_{R}$ and $I_{L}$ defined as above.  
We now set up the evaluation of $I_L$. 
We begin by writing
\[
   \frac{\zeta'}{\zeta}(s) \zeta'(s) A(s)= \sum_{m=1}^{\infty}
     a(m) m^{-s} 
\]
where 
\[
   a(m)= \sum_{uvw=m} \Lambda(u) \log(v) x_w \ . 
\]
It thus follows 
\[
   I_{L} = \sum_{k \le M} \frac{y_k}{k}
   \sum_{m=1}^{\infty} a(m) 
   \frac{1}{2 \pi i} 
    \int_{\kappa+i}^{\kappa+iT}  \chi(1-s) 
    \left(m/k\right)^{-s}
   \, ds    
\]
and we invoke 
\begin{lem} \label{stph} Let $r, \kappa_0 >0$ we have 
\[
    \frac{1}{2 \pi i} \int_{\kappa+i}^{\kappa+iT} \chi(1-s) r^{-s} \, ds
    = \delta(r) e(-r) + E(r,c)r^{-\kappa}
\]
uniformly for $\kappa_0 \le \kappa \le 2$ where $\delta(r) =1 $ if $r \le T/ 2 \pi$ 
and $\delta(r)=0$ otherwise.  Moreoever, 
\[
    E(r,\kappa) \ll T^{\kappa-1/2}+ \frac{T^{\kappa+1/2}}{|T-2 \pi r|+T^{1/2}} \ . 
\]
\end{lem}
\noindent This result follows from Lemma 2 of \cite{Go}. 
Applying Lemma \ref{stph} yields
$
   I_L = \mathcal{M} + T^{1/2}\mathcal{E}_{1}' + T^{3/2}\mathcal{E}_{2}'
$
where
\begin{equation}
   \mathcal{M}
   = \sum_{k \le M} \frac{y_k}{k} \sum_{m \le \frac{kT}{2 \pi}} a(m)
   e \left( - \frac{m}{k} \right)  \ , 
   \label{eq:M}
\end{equation}   
\begin{equation}
  \mathcal{E}_{1}' =  T^{1/2}
 \sum_{k \le M} |y_k|
 \sum_{m=1}^{\infty} \frac{|a(m)|}{m^{\kappa}}
   \ , 
\end{equation}
\begin{equation}
     \mathcal{E}_{2}' =  T^{3/2}
 \sum_{k \le M} |y_k|
 \sum_{m=1}^{\infty} \frac{|a(m)|}{m^{\kappa}}
  (|T-2 \pi m/n|+T^{1/2})^{-1} \ . 
\end{equation}
Note that 
\[
  \mathcal{E}_{1}' \ll 
  ||{\bf y}||_1 
  \sum_{m \le M} \frac{|x_m|}{m}
  \frac{\zeta'}{\zeta}(\kappa) \zeta'(\kappa)
  \ll \Lo^3  ||{\bf y}||_1   || \mbox{$ \frac{x_m}{m}$}||_1
  \ . 
\]
We next consider $\mathcal{E}_{2}'$.  We split this into the cases:
$(i)$ $|T-2\pi m/n|>T/2$, $(ii)$ $\sqrt{T} \le 
|T-2 \pi m/n| \le T/2$, and $(iii)$ $|T-2 \pi m/n| \le \sqrt{T}$. 
In case $(i)$, $g_{m,n}(T) \ll T^{-1}$ and we have 
\[
   \mathcal{E}_{21}' \ll T^{-1} ||{\bf y}||_1 
   \sum_{m=1}^{\infty} \frac{|a(m)|}{m^{\kappa}} \ll
   T^{-1} \Lo^3  ||{\bf y}||_1  || \mbox{$ \frac{x_m}{m}$}||_1
  \ . 
\]
In case $(ii)$ we begin by assuming without loss of generality that
$\sqrt{T} \le 2 \pi m/n-T \le T/2$.  We divide this into $\ll \log T$ 
intervals of the shape $T+P < 2 \pi m/n< T+2P$ with $\sqrt{T}
\ll P \ll T$.  We denote the interval $I= [\frac{nT}{2 \pi}
+ \frac{nP}{2 \pi}, \frac{nT}{2 \pi}
+ \frac{2nP}{2 \pi}]$. 
Note that $|a(m)| \le ||{\bf x}||_{\infty} \tau(m) \log^2 m$
and hence
\begin{equation}
\begin{split}
    \mathcal{E}_{22}' & 
    \ll \xinf \sum_{P} \sum_{n \le M} |y_n|
    \sum_{m \in I} \frac{\tau(m) \log^2 m}{mP} \\
    &  \ll T^{-1} \xinf \sum_{P} \sum_{n \le M} \frac{|y_n|}{nP}
    \sum_{m \in I} \tau(m) \log^2 m  \\
    & \ll  T^{-1} \xinf \sum_{P} \sum_{n \le M} \frac{|y_n|}{nP}
    (nP) \Lo^{3}
    \ll T^{-1} \Lo^{4} \xinf \yone \ . 
    \nonumber
\end{split}
\end{equation}
In the second inequality above we apply an estimate
for the divisor sum in short intervals. 
For a precise statement, see Lemma \ref{sidiv} which occurs later in the 
article. 
In the last case we have $|T-2 \pi m/n| \le \sqrt{T}$ and
$g_{m,n}(T) \ll T^{-1/2}$.  Put $J= [\frac{n}{2\pi}(T-\sqrt{T}) ,
\frac{n}{2 \pi}(T+\sqrt{T})]$. 
We now have 
\begin{equation}
\begin{split}
    \mathcal{E}_{23}' &
    \ll T^{-1/2} \xinf \sum_{n \le M} |y_n|
    \sum_{m \in J} \frac{\tau(m) \log^2 m}{m}
    \ll T^{-3/2} \xinf \sum_{n \le M} \frac{|y_n|}{n}
    (n \sqrt{T}) \Lo^3 \\
    &  \ll T^{-1} \Lo^3 \xinf \yone  \ . 
    \nonumber
\end{split}
\end{equation}
Combining our estimates yields $T^{1/2} \mathcal{E}_{1}'
+T^{3/2} \mathcal{E}_2'$ is bounded by 
\[
   T^{1/2} (\Lo^3 \yone  ||x_m/m||_1
   + \Lo^4 \xinf \yone)
   \ll T^{1/2} \Lo^{4} \yone \xinf
\]
and hence
\begin{equation}
    S   = S_{R} - \overline{\mathcal{M}} + O(
     T^{1/2} \Lo^{4} \yone \xinf)  
     \label{eq:Sid}
\end{equation}
where $S_{R}$ and $\mathcal{M}$ are given by~(\ref{eq:SR0}) 
and~(\ref{eq:M}) respectively.
We now simplify our expression~(\ref{eq:M}) for $\mathcal{M}$. 
The first step will be to express the additive character $e(-m/k)$ in terms
of multiplicative characters.  In order to do this we write 
$m/k= m'/k'$ with $(m',k')=1$.  We have the well-known 
identity
\[
    e \left(- \frac{m}{k} \right) = e \left( - \frac{m'}{k'} \right)
    =  \frac{1}{\phi(k')}
    \sum_{\chi \ \mathrm{mod} \ k'} \tau(\overline{\chi}) \chi(-m')
\]
where for a character $\chi$ modulo $k'$, 
$\tau(\chi) = \sum_{a=1}^{k'} \chi(a) e(a/k')$ is the usual Gauss sum.
Now note that $\tau(\chi_0)= \mu(k')$, where $\chi_0$ is the principal
character modulo $k'$.  Hence 
\begin{equation}
   e \left( - \frac{m}{k} \right)
   = \frac{\mu(k')}{\phi(k')}
   + \frac{1}{\phi(k')}  
       \sum_{{\begin{substack}{\chi \ \mathrm{mod} \ k'
         \\ \chi \ne \chi_0}\end{substack}}} 
   \tau(\overline{\chi}) \chi(-m') \ . 
   \label{eq:emk}
\end{equation}
The basic idea is that when we insert the expression~(\ref{eq:emk})
back in~(\ref{eq:M}) that $\mu(k')/\phi(k')$ term will account 
for the main term of $\mathcal{M}$ 
 and the sum over non-principal characters modulo
$k'$ will be an error term.  Before commencing with this strategy, 
we must first convert the above sum to a sum over primitive characters. 
This is since we shall invoke an analytic version of the large sieve inequality
involving only primitive characters. 
 If a character $\chi$ modulo $k'$ is 
induced by the primitive character $\psi$ modulo $q$ then we
have $\tau(\chi)= \mu(k'/q)\psi(k'/q) \tau(\psi)$ (see \cite[p.\ 67]{Da}). 
We shall use the notation 
$\sum_{\psi\mathrm{\mod} q}^{*}$ to denote summation over primitive characters modulo $q$.  Therefore
\[
   \frac{1}{\phi(k')}  
       \sum_{{\begin{substack}{\chi \ \mathrm{mod} \ k'
         \\ \chi \ne \chi_0}\end{substack}}} 
   \tau(\overline{\chi}) \chi(-m') 
   = \frac{1}{\phi(k')}
   \sum_{{\begin{substack}{q \mid k'
         \\ q > 1}\end{substack}}} 
   \ \chiq \mu \left( \frac{k'}{q} \right)
   \overline{\psi} \left(  \frac{k'}{q}
   \right) \tau(\overline{\psi})  \psi(-m') 
\] 
since $(m',k')=1$ implies $(m',q)=1$ and thus $\chi(-m')= \psi(-m')$.
The next step is to rewrite this formula in terms of $m$ and $k$. 
Let $g=(m,k)$.  By the M\"{o}bius inversion formula, we have 
\[
   f \left( m',k' \right)
   = f \left( 
   \frac{m}{g}, \frac{k}{g}
   \right) = \sum_{d \mid g} \sum_{e \mid d} \mu
   \left(  \frac{d}{e} \right)
   f \left( 
   \frac{m}{e}, \frac{k}{e} 
   \right) 
\]
for any function $f$.  Moreover, note that the condition 
$d \mid g$ is equivalent to $d \mid m$, $d \mid k$.  
Thus we derive 
\begin{equation}
\begin{split}
   &  \frac{1}{\phi(k')}  
       \sum_{{\begin{substack}{\chi \ \mathrm{mod} \ k'
         \\ \chi \ne \chi_0}\end{substack}}} 
   \tau(\overline{\chi}) \chi(-m')   \\
   & =   \sum_{{\begin{substack}{d \mid m
         \\ d  \mid k}\end{substack}}} 
         \sum_{e \mid d} \frac{\mu(d/e)}{\phi(k/e)}
          \sum_{{\begin{substack}{q \mid k/e
         \\ q > 1}\end{substack}}} 
          \ \chiq
         \mu \left( \frac{k}{eq} \right)
         \overline{\psi}  \left( \frac{k}{eq} \right)
        \tau(\overline{\psi}) \psi \left( - \frac{m}{e} \right) \\
   & =  \sum_{{\begin{substack}{q \mid k
         \\ q > 1}\end{substack}}}   \chiq \tau(\overline{\psi})
         \sum_{{\begin{substack}{d \mid m
         \\ d  \mid k}\end{substack}}} 
         \sum_{{\begin{substack}{e \mid d
         \\ e  \mid k/q}\end{substack}}}  \frac{\mu(d/e)}{\phi(k/e)}
         \overline{\psi} \left( 
         - \frac{k}{eq}  \right) \psi \left( \frac{m}{e} \right)
         \mu \left(
         \frac{k}{eq}
         \right)  \\
      & =  \sum_{{\begin{substack}{q \mid k
         \\ q > 1}\end{substack}}}  \ \chiq \tau(\overline{\psi})
          \sum_{{\begin{substack}{d \mid m
         \\ d  \mid k}\end{substack}}} 
         \psi \left( \frac{m}{d} \right) \delta(q,k,d,\psi)
         \label{eq:nonp}
\end{split}
\end{equation}
where 
\begin{equation}
      \delta(q,k,d,\psi) = 
       \sum_{{\begin{substack}{e \mid d
         \\ e \mid k/q}\end{substack}}} 
      \frac{\mu(d/e)}{\phi(k/e)} \overline{\psi}
      \left( - \frac{k}{eq}
      \right) \psi \left(\frac{d}{e} \right) \mu \left(
      \frac{k}{eq}
      \right) \ .
      \label{eq:delta}
\end{equation}
By~(\ref{eq:M}),~(\ref{eq:emk}), and~(\ref{eq:nonp}) we have now shown that
$\mathcal{M}= \mathcal{M}_0 + \mathcal{E}$ where 
\begin{equation}
    \mathcal{M}_0
    = \sum_{k \le M} \frac{y_k}{k} \sum_{m \le \frac{kT}{2 \pi}} a(m)
    \frac{\mu(k/(m,k))}{\phi(k/(m,k))} \ , 
    \label{eq:M01}
\end{equation}
\begin{equation}
  \mathcal{E} = 
  \sum_{k \le M} \frac{y_k}{k} \sum_{m \le \frac{kT}{2 \pi}} a(m)
   \sum_{{\begin{substack}{q \mid k
         \\ q > 1}\end{substack}}}  \ \chiq \tau(\overline{\psi})
          \sum_{{\begin{substack}{d \mid m
         \\ d  \mid k}\end{substack}}} 
         \psi \left( \frac{m}{d} \right) \delta(q,k,d,\psi)
    \ . 
  \label{eq:E}
\end{equation}
Thus we conclude by~(\ref{eq:Sid}) and above decomposition of
$\mathcal{M}$ that 
\begin{equation}
       S   = S_{R} - \overline{\mathcal{M}_{0}}
       - \overline{\mathcal{E}} + 
        O(
     T^{1/2} \Lo^{4} \yone \xinf)  \ . 
        \label{eq:S}
\end{equation}
The remainder of the article will be devoted to computing
asymptotic expressions for $S_R$ and $\mathcal{M}_{0}$ and for providing
an upper bound for $\mathcal{E}$.  The evaluation 
of $S_R$ is straightforward and will be done in the next section.
The evaluation of $\mathcal{M}_{0}$ is also essentially elementary. 
The most involved part of the argument will be in bounding $\mathcal{E}$. 
In fact, we shall establish the following results which will imply our theorem:
\begin{prop} \label{SR} We have
\[
     S_{R} =    \frac{T}{2 \pi} 
    \sum_{nu \le M} \frac{r_0(n)x_{u}y_{nu}}{nu}
    + O(  T^{\epsilon}(||{\bf y}||_{\infty} M
    + ||{\bf y}||_1) )  
\]
where $r_0(n) = P_2(\Log)-
    2P_1(\Log)(\log n)+  (\Lambda*\log)(n)$
and $P_2$, $P_1$ are monic polynomials of degrees $2, 1$ respectively.
\end{prop}
\begin{prop} \label{M0}
$(i)$ We have in the divisor case
\begin{equation}
   \mathcal{M}_{0} =
   \frac{T}{2 \pi}  \sum_{{\begin{substack}{u,v \le M
         \\ (u,v)=1}\end{substack}}}
   \frac{c(u,v) H(M;u,v)}{uv}  
   + \frac{T}{4 \pi} \sum_{gv \le M} \frac{y_{g}x_{gv}}{gv} 
    R_2 \left(\log(\mbox{$\frac{T}{2 \pi v}$}) \right)  + \mathcal{E}_0   
\end{equation}
where
\begin{equation}
   \mathcal{E}_0 \ll   T  \exp(-c_4 \sqrt{\log T}) 
  || \mbox{$ \frac{x_n}{n}$}||_1  ||
  \mbox{$\frac{(\tau_3*|x|)(n)y_n}{n} $} ||_{1}  
   \ , 
   \label{eq:E01}
\end{equation}
$H(M;u,v) = \sum_{g \le \min(M/u,M/v)} \frac{y_{ug}x_{vg}}{g}$,
\[
    c(u,v) =  - \mbox{$\frac{1}{2}$}\Lambda_2(a)
   + R_1 \left(  \log(\mbox{$\frac{T}{2 \pi v}$})  \right) \Lambda(u)
    + \tilde{R}_1 \left( 
   \log(\mbox{$\frac{T}{2 \pi v}$})
   \right)  \alpha_1(u) + \alpha_2(u)  \ ,
\]
and $R_1,\tilde{R}_1,R_2$ are monic polynomials of degrees $1,1,2$.  Moreover,
$\alpha_1, \alpha_2$ are arithmetic functions supported on those $n$ 
with $\omega(n) \le 2$.  More precisely,   $\alpha_1(p^{\alpha}) = \frac{\log p}{p-1}$, 
   $ \alpha_2(p^{\alpha})   =
  \mbox{$  - \frac{(\alpha+1) (\log p)^2}{p-1}
    + \frac{D\log p}{p-1} - \frac{\log p}{(p-1)^2}$}$
    for some $D \in \mathbb{R}$, and 
 $\alpha_2 (p^{\alpha} q^{\beta}) = \mbox{$- (\log p)(\log q) (
   \frac{1}{p-1} + \frac{1}{q-1})$}$ for $\alpha,\beta \in \mathbb{N}$.   \\
   
\noindent $(ii)$  In the resonator case, we have,  assuming the large zero-free region 
conjecture for the Riemann zeta function, 
the same result as above but with  
\begin{equation}
    \mathcal{E}_0 \ll   T  \exp 
    \left(-\mbox{$\frac{c'_4 \log T}{\log \log T}$} \right) 
    || \mbox{$ \frac{x_n}{n}$}||_1  
    || \mbox{$ \frac{(\tau_3*|x|)(n)y_n}{n}$} ||_{1}  \ . 
    \label{eq:E02}
\end{equation}
\end{prop}
\begin{thm} \label{Ebd}
$(i)$ If $x_n$, $y_n$ satisfy~(\ref{eq:divc}),~(\ref{eq:divc2}) then for $0< \theta < 1/2$ we
have for any $A'>0$
\begin{equation}
   \mathcal{E} \ll_{A'} T(\log T)^{-A'} +
   T^{\frac{3}{4}+\frac{\theta}{2}+\epsilon}  \ . 
   \label{eq:Ea}
\end{equation}
$(ii)$ Assume the large zero-free region conjecture. 
If $x_n=y_n=f(n)$ where $f$ is defined by
(\ref{eq:resc}) then for $0 < \theta < 1/6$
we have
\begin{equation}
    \mathcal{E} \ll T \exp \left(
    - \mbox{$\frac{c_5 \log T}{\log \log T}$}
    \right) + 
     T^{\frac{3}{4}+\frac{3\theta}{2}+\epsilon}  \ . 
     \label{eq:Eb}
\end{equation}
\end{thm}
By~(\ref{eq:S}), Propositions \ref{SR}-\ref{Ebd}, and the bounds for 
the coefficients given in section \ref{coeff}  we obtain
\begin{equation}
\begin{split}
  S & = \frac{T}{2 \pi}
   \sum_{nu \le M} \frac{x_{u}x_{nu}r_0(n)}{nu} 
   -\frac{T}{4 \pi} \sum_{gv \le M} \frac{y_g x_{gv}}{gv}
   R_{2} \left(
   \log(\mbox{$\frac{T}{2 \pi v}$})
   \right)
   \\
  & -
   \frac{T}{2 \pi}  \sum_{{\begin{substack}{u,v \le M
         \\ (u,v)=1}\end{substack}}}
   \frac{c'(u,v)H(M;u,v)}{uv}
   + \mathcal{E}_0 + \mathcal{E} 
    + O(  T^{\epsilon}(||{\bf y}||_{\infty} M
    + ||{\bf y}||_1) )  
   \nonumber
\end{split}
\end{equation}
where $\mathcal{E}_0$ and $\mathcal{E}$ are as in the proceeding propositions. 
Setting $r_1(u,v)=-c'(u,v)$ we see that we obtain the principal
term of Theorem \ref{thm1}.  By~(\ref{eq:E01}) and~(\ref{eq:divbds}) we obtain
an error term of the form $T(\log T)^{-A'}+ T^{\frac{3}{4}+ \frac{\theta}{2}+\epsilon}$ as asserted.  By~(\ref{eq:E02}) and~(\ref{eq:resb})
we obtain the error term $T \exp \left(
- \frac{c_1 \log T}{\log \log T}
\right) + T^{\frac{3}{4}+ \frac{\theta}{2}+\epsilon}$.
\end{proof}


\section{Evaluation of $S_R$: Proof of Proposition \ref{SR}}

In this section we evaluate the term $S_{R}$.  Recall that
\begin{equation}
  S_{R} = \frac{1}{2 \pi i} \int_{\kappa+i}^{\kappa+iT} \left(
  \frac{\chi'}{\chi}(s)^{2} \zeta(s) - 2 \frac{\chi'}{\chi}(s) \zeta'(s)
  + \frac{\zeta'}{\zeta}(s) \zeta'(s)
  \right) X(s) Y(1-s) \, ds \ .
  \nonumber
\end{equation}
The above integral will be evaluated by considering the more general expression:
\[
    J_k = J_k (T) = \frac{1}{2 \pi i} 
    \int_{\kappa+i}^{\kappa+iT} 
    \left( \frac{\chi'}{\chi}(s) \right)^{k} D(s) X(s) Y(1-s) \, ds 
\]
where $D(s)= \sum_{n =1}^{\infty} \alpha_n n^{-s}$ and $k \in 
\mathbb{Z}_{\ge 0}$.  Suppose that  
$
    \sum_{n=1}^{\infty} |\alpha_n|n^{-\sigma}
    \ll (\sigma-1)^{-\alpha}
$
as $\sigma \to 1$. 
We will establish:
\begin{lem} \label{Jk}
Suppose $|\alpha_n| \ll n^{\epsilon}$ and we have coefficients
$x_n,y_n$ satisfying  $||x_n/n||_{1},
||y_n/n||_{1} \ll T^{\epsilon}$.
 Then for $k \in \mathbb{N}$
\begin{equation*}
    J_k = \frac{(-1)^k TP_k(\Log)}{2 \pi}  
     \sum_{nu \le M} \frac{\alpha_{n}x_{u}y_{nu}}{nu}  
     +  O  \left(   T^{\epsilon}
   (||{\bf y}||_{\infty} M
    + ||{\bf y}||_1)  \
      \right) 
\end{equation*}
where $P_{k}$ is a monic polynomial of degree $k$. 
\end{lem}
\begin{proof}[Proof of Proposition \ref{SR}]
By  our expression for $S_R$ above it suffices to apply Lemma \ref{Jk} in the cases
$k=2, \alpha_n=1$,
 $k=1, \alpha_n= -(\log n)$, and $k=0, \alpha_n = (\Lambda*\log)(n)$. 
Thus
\[
   S_{R}  
   \sim  \frac{T}{2 \pi} \sum_{nu \le M} \frac{x_u y_{nu}}{nu}
   \left(
   P_{2}(\Log) -2P_{1}(\Log)(\log n)+  (\Lambda*\log)(n)
   \right) 
\]
with an error $ O(T^{\epsilon}( ||{\bf y}||_{\infty} M
    + ||{\bf y}||_1 )$ as claimed.  
\end{proof}

\begin{proof}[Proof of Lemma \ref{Jk}]
We have the estimate $\frac{\chi'}{\chi}(s) = - \log \frac{|t|}{2 \pi} + O(1/(1+|t|))$  valid for $1/2 \le \sigma \le 2$ and $t \ge 1$.   Thus
\[
    J_k = \frac{1}{2 \pi i} \int_{\kappa+i}^{\kappa+iT}
    ( - \log(t/2 \pi))^{k} + O_{k} \left( \Lo^{k-1} t^{-1} \right))
    D(s)X(s)Y(1-s) \, ds 
   \ . 
\]
One checks that the error term contributes 
$\ll T^{\epsilon} ||{\bf y}||_1 $. 
Exchanging summation and integration order yields
\[
   J_k = \sum_{n,u,v}  
   \frac{\alpha_{n} x_{u}y_{v} (-1)^k}{n^{\kappa}u^{\kappa} v^{1-\kappa} 2 \pi }
   \int_{1}^{T}
   (\log(t/2 \pi))^k \left( \frac{v}{nu} \right)^{it} dt
   + O \left(
  T^{\epsilon}
|| {\bf y} ||_1
   \right)
\]
where $s = \kappa +it$. 
We now  write
$J_{k}=J_{d}+ J_{nd}$ where $J_d$ consists of the 
the diagonal terms $v=nu$ and $J_{nd}$ consists of the terms
$v \ne nu$.  We have
\[
  J_{d} = \frac{(-1)^k}{2 \pi}
   \sum_{nu \le M} \frac{\alpha_{n}x_{u}y_{nu}}{nu} 
   \int_{1}^{T} \log^{k}(t/2 \pi) \, dt \ . 
\]
It is simple to see that
$
    \int_{1}^{T} \log^{k} \left( \frac{t}{2 \pi} \right) \, dt 
    = T P_{k} \left( 
    \Log
    \right) + O_{k}(1)
$
where $P_k$ is monic of degree $k$.  Moreover, 
since $|\alpha_n|,|| \frac{x_n}{n}||_1,
 || \frac{y_n}{n}||_1 \ll T^{\epsilon}$ we have
\[
    J_{d} = \frac{(-1)^kT P_{k}(\Log)}{2 \pi}
   \sum_{nu \le M} \frac{\alpha_{n}x_{u}y_{nu}}{nu} 
    + O \left( T^{\epsilon}  \right) \ . 
\]
The remainder term is 
\begin{equation}
  J_{nd} = \sum_{{\begin{substack}{n,u,v
         \\ v \ne nu}\end{substack}}}
  \frac{\alpha_{n}x_{u}y_{v}}{n^{\kappa}u^{\kappa} v^{1-\kappa}}
   \frac{1}{2 \pi} \int_{1}^{T} \log^{k}(t/2 \pi) \left( \frac{v}{nu} \right)^{it} dt
   \ . 
   \nonumber
\end{equation}
For $v \ne nu$ the integral is $ \ll  \Lo^{k}(|\log \frac{v}{nu}|)^{-1}$
and hence 
\[
   J_{nd} \ll \Lo^{k+\alpha} ||x_n/n||_{1}
         \sum_{{\begin{substack}{v \le M
         \\ v \ne nu}\end{substack}}} \frac{|y_v|}{v^{1-\kappa} |\log \left( \frac{v}{nu} \right)|} \ . 
\]
Since $\kappa=1+O(\Lo^{-1})$ it suffices to bound
$
  S(h) =  \sum_{{\begin{substack}{v \le M
         \\ v \ne h}\end{substack}}} |y_v||\log \frac{v}{h}|^{-1}
$.  If $h \ge 1.1M$, we have $S(h) \ll  ||{\bf y}||_1$.  We now suppose $h < 1.1M$.
The contribution to $S(h)$ from those $v \ge 1.5h$ and $v \le 0.5h$ is bounded
by $||{\bf y}||_1$.   Consider the interval $I=[0.5j,1.5j] \cap [1,M]$.
For those integers $k$ not in this interval we obtain 
\begin{equation}
   \sum_{k \notin I} \frac{|y_{k}|}{|\log(k/j)|}
   \ll  \sum_{k \le j/2} |y_{k}| + \sum_{3j/2 \le k \le M}  |y_{k}|
   \ll  ||{\bf y}||_{1} \ . 
   \nonumber
\end{equation}    
Either $I=[j/2,3j/2], [1,3j/2],[j/2,M]$.  In the first case 
\begin{equation}
\begin{split}
    & \sum_{k \in I} \frac{|y_k|}{|\log(j/k)|}
    =  \sum_{s=1}^{j/2} \frac{|y_{j-s}|}{|\log(j/(j-s))|}+
     \sum_{s=1}^{j/2} \frac{|y_{j+s}|}{|\log(j/(j+s))|} \\
    & \ll j\sum_{s=1}^{j/2} \frac{|y_{j-s}|}{s}
    + j \sum_{s=1}^{j/2} \frac{|y_{j+s}|}{s} 
    \ll ||{\bf y}||_{\infty} j (\log j) \ . 
    \nonumber
\end{split}
\end{equation}
The argument for the second and third cases is analogous.  We deduce 
$J_{nd} \ll   T^{\epsilon} (||{\bf y}||_{\infty} M
    + ||{\bf y}||_1 ) $ and thus the lemma is established. 
\end{proof}
\section{Preliminary lemmas}

In the section we prove several lemmas that will aid us in
evaluating $\mathcal{M}_0$ and bounding $\mathcal{E}$. 
Lemmas \ref{conv} and \ref{shu} will be applied in our evaluation of 
$\mathcal{M}_0$.  Lemmas \ref{conv}, \ref{shiu}, \ref{delta}, and \ref{theta} shall be invoked when we bound $\mathcal{E}$.   
The next lemma tells 
us how to decompose an arithmetic function which is the convolution of 
other arithmetic functions. 
\begin{lem} \label{conv}
Let $f_{1}, f_{2}, \ldots, f_{j}$ be arithmetic functions and let $D \in 
\mathbb{N}$.   Given a decomposition $D=d_1 d_2 \cdots d_j$ we define 
the integers $D_{i} = \prod_{u=1}^{j-i} d_u$ for $1 \le i \le j-1$ and
$D_j=1$. 
We have the identities
\begin{equation}
\begin{split}
   & 
      \sum_{{\begin{substack}{m \le X
         \\ (m,k)=1}\end{substack}}}
    (f_{1}*f_{2}* \cdots* f_{j})(mD)  \\
    & =
    \sum_{d_1 d_2 \cdots d_j = D}
        \sum_{{\begin{substack}{m_1 m_2 \cdots m_j \le X
         \\ (m_i,kD_i)=1}\end{substack}}}
        f_1(m_1 d_{j}) f_2 (m_2 d_{j-1}) \cdots
    f_j (m_j d_1)  \ , 
    \label{eq:id1}
\end{split}
\end{equation}
\begin{equation}
   \sum_{(m,k)=1} 
   \frac{ (f_{1}*f_{2}* \cdots* f_{j})(mD) }{m^s} 
   =  \sum_{d_1 d_2 \cdots d_j = D}
   \prod_{i=1}^{j} \sum_{(m_i,kD_i)=1} \frac{f_{i}(m_id_{j-i})}{m_{i}^{s}}
   \ . 
   \label{eq:id2}
\end{equation}
\end{lem}
\begin{proof}  The proof of this argument follows the proof 
of Lemma 3 of \cite{CGG3} (p. 506).  The case $j=2$ follows from 
the identity
\[
    (f_1*f_2)(mD) = \sum_{g \mid D}
      \sum_{{\begin{substack}{h \mid m \\
        (h,\frac{D}{g})=1}\end{substack}}} f_1(gh) f_2 \left(
        \frac{mD}{gh}
        \right)  \ . 
\]
By making the identifications $g=d_2,\frac{D}{g}=d_1,h=m_1$, and $\frac{m}{h}=m_2$
we obtain~(\ref{eq:id1}), ~(\ref{eq:id2}) for $j=2$.  For $j>2$ the assertion follows by induction.  
\end{proof}
 
We now introduce some arithmetic functions that will appear in our
evaluation of $\mathcal{M}_0$.  We define $\eta_1(k)= \sum_{p \mid k} \frac{\log p}{p-1}$, $\eta_2(k)
= -\sum_{p \mid k} \frac{p \log p}{(p-1)^2}$, 
\begin{equation}
\begin{split}
    g (h,k)   & =  \sum_{{\begin{substack}{a \mid h, a=p^t
         \\ (a,k)=1}\end{substack}}} \frac{\Lambda(a) \log p}{p-1} \ . 
     \label{eq:gj}  
\end{split}
\end{equation}
Moreover, we define arithmetic functions $\phi_{j}(n)$ for $j=1 \ldots
4$ as follows:
\begin{equation}
\begin{split}
    \phi_1 (n)  = \sum_{k \mid n} \mu(k) \eta_1 (k)  \ , \
   & \phi_2 (n)  = \sum_{k \mid n} \mu(k) \eta_2 (k)  \ , \\
    \phi_3 (n)  = \sum_{k \mid n} \mu(k)  g(h,k)  \ ,  \ 
   &  \phi_4(n)  = \sum_{k \mid n} \mu(k) \eta_1(k) \log k  \  .  
    \label{eq:phij}
\end{split}
\end{equation}
We prove that the $\phi_j$ are supported on integers $n$ such that 
$\omega(n) \le 2$. 
\begin{lem} \label{arith}
$\phi_1, \phi_2, \phi_3$ are supported only on the prime powers and
are given by:
\begin{equation}
\begin{split}
   & \phi_1(p^{\alpha})= - \frac{\log p}{p-1} \ , \
   \phi_2(p^{\alpha}) = \frac{p \log p}{(p-1)^2} \ , \ 
  \phi_3(p^{\alpha}) = \frac{\log p}{p-1} \log(p^{\alpha})
  \nonumber
\end{split}
\end{equation}
for $\alpha \ge 1$.  However, $\phi_4$ is supported 
on those $n$ with $\omega(n) \le 2$ and is given by
\begin{equation}
\begin{split}
    \phi_4(p^{\alpha} q^{\beta})
    & = (\log p)(\log q) \left( 
    \frac{1}{p-1} + \frac{1}{q-1}
    \right)  \ , \
    \phi_4(p^{\alpha}) = - \frac{(\log p)^2}{p-1}  \\
        \nonumber       
\end{split}
\end{equation}
where $\alpha,\beta \ge 1$. 
\end{lem}
\begin{proof}  
These formulae for the $\phi_i$
may be proved directly from their definitions, however it is simpler to employ
generating functions.  Put $A(s) = \sum_{n=1}^{\infty} \phi_4(n)
n^{-s}$ where 
$
   \phi_4(n) = \sum_{hk=n} \mu(k) (\sum_{p \mid k} f(p)) \log k 
$
and $f(x)= \frac{\log x}{x-1}$.  We have
\begin{equation}
\begin{split}
  A(s) & = \sum_{h,k \ge 1} \frac{\mu(k) \log k}{h^s k^s}
  \sum_{p \mid k} f(p) 
  = \sum_{p} f(p) 
  \sum_{h=1}^{\infty} \frac{1}{h^s}
  \sum_{p \mid k} \frac{\mu(k) \log k}{k^s} \\
  & = \sum_{p} \frac{f(p) \mu(p)}{p^s} 
  \zeta(s) \sum_{(k_{1},p)=1} \frac{\mu(k_1)  \log(pk_1)}{k_{1}^s}  \\
  & = \sum_{p} \frac{f(p) \mu(p)}{p^s} 
 \zeta(s) 
 \left(
 \frac{\log p}{1-\frac{1}{p^s}} \zeta(s)^{-1}
 +  \sum_{(k_{1},p)=1} \frac{\mu(k_1)  \log(k_1)}{k_{1}^s}  
 \right) \ . 
 \nonumber 
\end{split}
\end{equation}
It may be verified that
\[
     \sum_{(k_{1},p)=1} \frac{\mu(k_1)  \log(k_1)}{k_{1}^s}  
     = (1-1/p^{s})^{-2} \frac{\log p}{p^s} \zeta(s)^{-1}
     + (1-1/p^{s})^{-1} \frac{\zeta'(s)}{\zeta(s)^2} 
\]
and thus
\begin{equation}
\begin{split}
  & A(s) = \sum_{p} \frac{f(p) \mu(p)}{p^s} 
  \left(
   \frac{\log p}{(1-1/p^{s})}
  + \frac{\log p}{p^s(1-1/p^s)^{2}}
  + \frac{1}{(1-1/p^{s})} \frac{\zeta'(s)}{\zeta(s)} 
  \right)  \\
  & = \sum_{p^j; j \ge 1} \frac{f(p) \mu(p) \log p}{p^{js}}
  + \sum_{p^j; j \ge 2} \frac{f(p) \mu(p) \log p (j-1)}{p^{js}} 
  + \sum_{p^j} \frac{f(p) \mu(p)}{p^{js}} \frac{\zeta'(s)}{\zeta(s) } \\
  & = \sum_{p^j} \frac{f(p) \mu(p) \log(p^j)}{p^{js}}
    + \left( \sum_{p^j} \frac{f(p)}{p^{js}} \right)
    \left( - \frac{\zeta'(s)}{\zeta(s) }  \right) \ . 
    \nonumber
\end{split}
\end{equation}
Let $\theta(n)$ be supported on prime powers defined by 
$\theta(p^j) = f(p)$.  We write 
\[
    \left( \sum_{p^j} \frac{f(p)}{p^{js}} \right)
    \left(  - \frac{\zeta'(s)}{\zeta(s) }  \right)
     = \sum_{n \ge 1} \frac{t(n)}{n^s}
\]
with $t(n) = (\theta*\Lambda)(n)$.  It follows easily that 
$\phi_4$ is supported on those $n$ such that $\omega(n) \le 2$. 
We begin
by supposing that $\omega(n)=2$ and $n=p^{\alpha} q^{\beta}$.
Thus 
\begin{equation}
  t(p^{\alpha} q^{\beta}) = \sum_{uv=p^{\alpha} q^{\beta}}
  \Lambda(u) \theta(v) 
  =  (\log p) f(q) 
  + (\log q) f(p) \ .
  \label{eq:t1}
\end{equation}
Next for $n=p^{\alpha}$
\begin{equation}
  t(p^{\alpha}) 
  = \sum_{uv = p^{\alpha}} \Lambda(u) \theta(v) 
  =  \sum_{i+j=\alpha, i,j \ge 1} \log p f(p) 
   = (\alpha-1) f(p) (\log p) \ . 
    \label{eq:t2}
\end{equation}
Since $\phi_4(n)=f(n) \mu(n) \log(n) + t(n)$ the result follows from 
(\ref{eq:t1}) and (\ref{eq:t2}). 
\end{proof}

The next result provides an estimate for divisor sums in short intervals.
This is Theorem 2 of \cite{Sh}. 
\begin{lem} \label{sidiv}
Let $\alpha, \beta$ be real numbers and let $a,q,k$ be integers.
Suppose that $0< \alpha ,\beta < 1/2$, $k \ge 2$, and $(a,k)=1$.
We have as $t \to \infty$,
\[
      \sum_{{\begin{substack}{t-u \le n \le t
         \\ n \equiv a (q)}\end{substack}}}
   \tau_k (n) \ll \frac{u}{q}
     \left(\frac{\phi(q)}{q}\log t \right)^{k} 
\]
uniformly in $a,q$, and $t$ provided that $q < t^{1-\alpha}$
and $t^{\beta} < u \le t$. 
\end{lem}
We will also require a short interval estimate for $\tau_k*x$ where
$x$ is an arbitrary arithmetic function.
\begin{lem} \label{shiu}
Let $T \ll w \ll T^2$ and $M \le \sqrt{T}$. 
Let $x$ be an arithmetic function supported on $[1,M]$. Then 
\[
    \sum_{t-u \le n \le t} (\tau_{k}*x)(n)
    \ll u (\log t)^{k-1} ||x_n/n||_{1}
\]
where $u=wU^{-1}$, $\exp(c \sqrt{\log w}) \le U \le \exp(\frac{2\log w}{\log \log w})$ and 
$
   \frac{w}{2} \le  t-u \le t \le w 
$.
\end{lem}
\begin{proof}
Notice that our sum may be rewritten as 
\[
    \sum_{t-u < n \le t} (\tau_k*x)(n)
    =  \sum_{b \le M} x_{b} \sum_{\frac{t}{b}-\frac{u}{b} \le a \le \frac{u}{b}}
    \tau_{k}(b) \ . 
\]
By our conditions on $u$ and $t$, $\frac{u}{b} \gg \left( \frac{t}{b} \right)^{\epsilon}$ for all $1 \le b \le M$. 
Hence by Lemma \ref{sidiv}
\[
    \sum_{b \le M} x_{b} \sum_{\frac{t}{b}-\frac{u}{b} \le a \le \frac{t}{b}}
    \tau_{k}(a) 
    \ll u (\log t)^{k-1}
    \sum_{b \le M}\frac{x_b}{b}
    \ll u \Lo^{k-1} ||\mbox{$ \frac{x_n}{n}$}||_{1} \ . 
\]
\end{proof}
The next lemma evaluates asymptotically a sum that will appear in 
our evaluation of $\mathcal{M}_0$.  The proof of this lemma is very
similar to the asymptotic evaluation of $\sum_{n \le x} \Lambda(n)$
that occurs in the proof of the prime number theorem. 
\begin{lem} \label{shu}
Let $h,k \in \mathbb{N}$, $h,k \le M$  and $x \ge 1$
such that $\log x \asymp \log T$. 
\begin{equation}
\begin{split}
    \sum_{{\begin{substack}{u \le x
         \\ (u,k)=1}\end{substack}}} (\Lambda*\log)(hu)& = \frac{x \phi(k)}{k}
   \left( 
   \frac{1}{2} (\log x)^2 + 2\log(x/e) \log h + (\Lambda*\log)(h) 
   \right. \\
    &  \left. +(C_{0} - \eta_1(k)) \log(x/e)
    +C_1 \eta_1(k)-\eta_2(k)-g(h,k) \right)  \\
    & + O(\tau(h) x \exp(-C_{2} \sqrt{\log T})) 
    \nonumber
\end{split}
\end{equation}
for some explicit constants $C_{j}$ for $j=0,1,2$ and 
$\eta_1,\eta_2,g_1,g_2$, and $g$  are the arithmetic functions 
defined above.  Note if we further assume that $\zeta(s)$ is 
non-vanishing in the region $\mathrm{Re}(s)
 \ge 1 - \frac{c_0}{\log \log (|\mathrm{Im}(s)|+4)}$
then the above error term may be reduced to 
\[
    O \left(  \tau(h) x \exp \left( -
    \mbox{$\frac{ C'_2\log T}{\log \log T}$} \right)
    \right) \ . 
\]
\end{lem}
\begin{proof} 
Put $A(z) = \sum_{(u,k)=1} (\log*\Lambda)(hu)u^{-z}$.
We have by Perron's formula 
\[
      \sum_{{\begin{substack}{u \le x
         \\ (u,k)=1}\end{substack}}} (\Lambda*\log)(hu) = \frac{1}{2 \pi i}
         \int_{\kappa-iU}^{\kappa+iU}  A(z)x^z \frac{dz}{z} 
         + O \left( x^{\kappa}
         \sum_{n=1}^{\infty} \frac{ (\Lambda*\log)(hn)}{n^{\kappa}
         (1+ T |\log(x/n)|)}
          \right)
\]
for $\kappa= 1+ O((\log x)^{-1})$  (see \cite[p.\ 132]{Te}).  
First note that $(\Lambda*\log)(hn) \ll \tau(h)^2 (\log n)^2 \tau(n)^2$.
In the last sum above the contribution
from those $n$ not in  $[0.5x,2x]$ is 
\[
  \ll \frac{\tau(h)x}{U}  \sum_{n \ge 1} \frac{(\log n)^2 \tau(n)^2}{n^{\kappa}}
  \ll \frac{\tau(h) x (\log x)^{6}}{U} \ . 
\]
We now consider the contribution from those $n \in [0.5x,2x]$.  We begin with 
the interval $[0.5x,x)$.  This yields the contribution 
\begin{equation}
\begin{split}
   & \ll \tau(h) \Lo^2
   \sum_{\frac{x}{2} \le n < x} \tau(n) \min(1, U^{-1} |\log(x/n)|^{-1})  \ . 
   \nonumber
\end{split}
\end{equation}
Since $|\log(x/n)|^{-1} \ll \frac{x}{|x-n|}$ this sum is 
\begin{equation}
\begin{split}
   \ll \sum_{x-\frac{x}{U} \le n < x} \tau(n)
   + \frac{x}{U} \sum_{\frac{x}{2} < n < x - \frac{x}{U}} \frac{\tau(n)}{x-n}
   \ . 
   \nonumber
\end{split}
\end{equation}
By an application of Lemma \ref{sidiv} the first sum is $\ll \frac{x}{U} \log x$.
Now dividing the second sum into $K \ll U$ intervals of length $\frac{x}{U}$
and invoking again Lemma \ref{sidiv} we see that
\begin{equation}
\begin{split}
    &  \frac{x}{U} \sum_{\frac{x}{2} < n < x - \frac{x}{U}} \frac{\tau(n)}{x-n}
     \ll \frac{x}{U}
     \sum_{k=1}^{K} \frac{U}{kx}  \sum_{x-(k+1)\frac{x}{U} \le n <
     x-k \frac{x}{U}} \tau(n)  \\
     & \ll  \sum_{k=1}^{K} k^{-1}  \frac{x \log x}{U}
     \ll \frac{x \log x \log U}{U} \ . 
     \nonumber
\end{split}
\end{equation}
Combining estimates we deduce for $U \le x$
\[
      \sum_{{\begin{substack}{u \le x
         \\ (u,k)=1}\end{substack}}} (\Lambda*\log)(hu) = \frac{1}{2 \pi i}
         \int_{\kappa-iU}^{\kappa+iU}  A(z)x^z \frac{dz}{z} 
         +  O \left( 
         \frac{\tau(h) x \Lo^6}{U}
         \right) \ . 
\]
In a moment we shall give a decomposition of $A(z)$ 
in terms of other well known Dirichlet series and thus we shall show that it
has a triple pole at $z=1$.  Let $\sigma_0 (t) = 1-\frac{c'}{\log(|t|+2)}+it$
for $t \in \mathbb{R}$. We shall shift the contour left to 
$\mathrm{Re}(s)= \sigma_{0}(U)$.  Therefore
\begin{equation}
\begin{split}
   &   \sum_{{\begin{substack}{u \le x
         \\ (u,k)=1}\end{substack}}} (\Lambda*\log)(hu) 
       =  \mathrm{Re}_{z=1} \left( A(z)x^z z^{-1} \right) +  
       \frac{1}{2 \pi i} \int_{\sigma_{0}(U)-iU}^{\sigma_{0}(U)+iU}
       A(z) x^z \frac{dz}{z}  \\
   & + O \left(
   \left( \int_{\kappa+iU}^{\sigma_0(U)+iU} + \int_{\kappa-iU}^{\sigma_0(U)-iU}
   \right)
    A(z)x^z \frac{dz}{z}
    + \frac{\tau(h) x \Lo^6}{U} 
   \right)  \ . 
   \label{eq:cint}
\end{split}
\end{equation}
By Lemma \ref{conv}
\[
   A(z) = \sum_{ab =h} \sum_{(c,ak)=1} \frac{\log(bc)}{c^z}
   \sum_{(d,k)=1} \frac{\Lambda(ad)}{d^z} :=
   \sum_{ab=h} A_1(z;a,b) A_2(z;a) \  .
\]
A calculation shows that
\[
   A_1(z;a,b) =
   \Phi(z;ak) \left(
   \log(b) \zeta(z)-\zeta'(z) - \zeta(z) \eta(z;ak)
   \right)
\]
where for $n \in \mathbb{N}$ 
\[
     \Phi(z;n) = \prod_{p \mid n} \left(1 - p^{-z} \right)
      \ , \
      \eta(z;n) := 
      \frac{\Phi^{'}(z;n)}{\Phi(z;n)} = 
      \sum_{p \mid n} \frac{\log p}{p^z-1} \ . 
\]
Also 
\[
    A_2 (z;a) = \left\{ \begin{array}{ll}
                   - \frac{\zeta'}{\zeta}(z) -\eta(z;k)
                   & \mbox{if $a=1$} \\
                 \frac{\log p}{1-p^{-z}} & \mbox{if} \ a=p^l \ , \ (a,k)=1    \\
                 \log p & \mbox{if} \ a=p^l \ , \ p \mid k \\
                 0 & \mbox{else}              
                      \end{array} \ . 
          \right. 
\]
It is convenient to define $\tilde{\Lambda}(a;z)= (\log p)/(1-p^{-z})
= \Lambda(a)(1+ \frac{1}{p^z-1})$ for $a=p^l$.  Thus we have
$A(z)= B_1(z)+B_2(z)+B_3(z)$ where 
\begin{equation}
  B_1(z)   = \left(
   - \frac{\zeta'}{\zeta}(z) - \eta(z;k)
   \right)  \Phi(z;k)\left(
   -\zeta'(z) + \zeta(z)(\log h - \eta(z;k)
   \right)  \ , 
\end{equation}
\begin{equation}
  B_2(z)  = \sum_{a \mid (h,k)}
   \Lambda(a) \Phi(z;ak)
   \left(
   - \zeta'(z) +\zeta(z)(\log(h/a)- \eta(z;ak))
   \right)   \ , 
\end{equation}
\begin{equation}
  B_3(z) = 
   \sum_{{\begin{substack}{a \mid h
         \\ (a,k)=1}\end{substack}}}
   \tilde{\Lambda}(a;z) \Phi(z;ak)
   \left(
   - \zeta'(z) +\zeta(z)(\log(h/a)- \eta(z;ak))
   \right)   \ . 
\end{equation}
We define $R_{j}= \mathrm{Res}_{z=1} ( B_j(z)x^{z}/z)$ for $j=1,2,3$.
We remark that $\eta(1;k)=\eta_1(k)$, $\eta'(1;k)=\eta_2(k)$, $\Phi(1;k) = \frac{\phi(k)}{k}$,
\begin{equation}
\begin{split}
   \Phi'(1;k) & = \Phi(1;k) \eta_1(k) = \frac{\phi(k)}{k} \eta_1(k) 
   \ ,  \\
   \Phi^{(2)}(k) 
   & = \Phi(1;k) \left(
   \eta(1;k)^2 + \eta'(1;k)
   \right) = \frac{\phi(k)}{k} 
   \left(
   \eta_1(k)^2 + \eta_2 (k) 
   \right) \ .
   \label{eq:phiid} 
\end{split}
\end{equation}
We now list several Laurent series that we require in the residue computation
\begin{equation}
\begin{split}
     \frac{x^z}{z}
    & = x \left(
    1 + 
    \log(x/e)
    (z-1)
    + \left(
    \frac{1}{2} (\log x)^2 - \log (x/e)
    \right)(z-1)^2 + \cdots
    \right) \ ,  \\
    \Phi(z;n)  & = \Phi(1;n) + \Phi^{'}(1;n)(z-1)
   + \frac{1}{2} \Phi^{(2)}(1;n) (z-1)^2 + \cdots  \ , \\
    \frac{\zeta'(z)^2}{\zeta(z)}
    & = \frac{1}{(z-1)^3}
   \left(
   1 + a_1 (z-1) + a_2 (z-1)^2 + \cdots 
   \right)  \ ,   \\
    -\zeta'(z)
   & = \frac{1}{(z-1)^2}
   \left(
   1 + b_2 (z-1)^2 + \cdots 
   \right)   \ , 
   \nonumber
\end{split}
\end{equation}
with $a_j,b_j \in \mathbb{R}$.  
Note that 
\[
    B_1(z) = \Phi(z;k) \left(  \frac{\zeta'(z)^2}{\zeta(z)}
    - \zeta'(z) \left( \log(h) -2\eta(z;k) \right)
     + \zeta(z)(\eta(z;k)^2 -\log(h)\eta(z;k) ) \right) \ . 
\]
We begin by writing $R_1 = R_{11}+R_{12}+R_{13}$ where
\begin{equation}
\begin{split}
  R_{11} & = \mathrm{Res}_{z=1}
  \left(
  \Phi(z;k) \frac{\zeta'(z)^2}{\zeta(z)} \frac{x^z}{z}
  \right) \ , \\
   R_{12}  & = \mathrm{Res}_{z=1} \left(
     -\zeta'(z)  \Phi(z;k) \left( \log(h)-2\eta(z;k) \right)
     \frac{x^z}{z}
     \right)  \ , \\
     R_{13} & =  \mathrm{Res}_{z=1}
   \left(
   \zeta(z) \Phi(z;k) (\eta(z;k)^2 - \log(h) \eta(z;k)) \frac{x^z}{z}
   \right)    \ . 
   \nonumber  
\end{split}
\end{equation}
We deduce from the above Laurent series that
\begin{equation}
\begin{split}
 R_{11} & = x ( (1/2) \log(x)^2 \Phi(1;k)
 + \log(x/e) ( (a_1-1)\Phi(1;k) + \Phi'(1;k)) \\
 & + (a_2 \Phi(1;k)+ a_1 \Phi'(1;k) + \Phi^{(2)}(1;k) ) \ . 
 \nonumber
\end{split}
\end{equation}
By (\ref{eq:phiid})  this simplifies to
\[
R_{11} =  \frac{x \phi(k)}{k} 
  \left( \frac{1}{2}\log(x)^2 + \log(x/e)( (a_1-1)+ \eta_1(k)) 
 + (a_2 + a_1 \eta_1 (k) + \eta_1(k)^2+\eta_2 (k)) \right)  \ . 
\]
Similar calculations yield
\[
   R_{12} = \frac{x \phi(k)}{k}
   \left(
   ( \log(x/e)+\eta_1 (k))(\log h - 2 \eta_1 (k))
   -2 \eta_2 (k)
   \right)
\]
and $R_{13}  
   = \frac{x\phi(k)}{k} ( \eta_1 (k)^2 - \eta_1 (k) \log h)$.
Combining our formulae we have
\begin{equation}
\begin{split}
  R_1 & = \frac{1}{2} (\log x)^2 + \log(x/e)( \log h + a_1 -1 -\eta_1 (k))
  + a_1  \eta_1 (k) -\eta_2 (k)  \ .
  \label{eq:R1}
\end{split}
\end{equation}
We next deal with $R_2$. Note that 
\[
  \mathrm{Res}_{z=1}
  \left(
  - \zeta'(z) \Phi(z;ak) \frac{x^z}{z}
  \right)
  = x \Phi(1;ak)
  \left( \log(x/e)  + \eta_1(ak) \right) \ , 
\]
\[
   \mathrm{Res}_{z=1}
   \left(
   \zeta(z)
   \Phi(z;ak)(\log(h/a) -\eta(z;ak)) \frac{x^z}{z}
   \right)
   = x \Phi(1;ak)
   \left( \log(h/a)-\eta_1 (ak) \right) \ , 
\]
and thus 
\[
   R_2 = 
   x \sum_{a \mid (h,k)} \Lambda(a) \frac{\phi(ak)}{ak}
   \left( 
    \log(x/e) + \log(h/a)  
   \right) \ . 
\]
Now observe that 
$\frac{\phi(ak)}{ak}= \frac{\phi(k)}{k}$ if $a \mid k$ and 
$\frac{\phi(ak)}{ak}= \frac{\phi(k)}{k}(1-1/p)$
if $(a,k)=1$ and hence
\begin{equation}
\begin{split}
   R_2 & = 
   \frac{x \phi(k)}{k}  \sum_{a \mid (h,k)} \Lambda(a)
   \left( 
   \log(x/e) +\log (h/a) 
   \right)  \ . 
\end{split}
\end{equation}
We now consider $R_3$.  Since
\begin{equation}
\begin{split}
   & \mathrm{Res}_{z=1}
  \left(
  -\zeta'(z) \tilde{\Lambda}(a,z) \Phi(z;ak)
  \right)  \\
  & = x \Phi(1;ak)
  \left(
  \tilde{\Lambda}(a;1)\log (x/e)+ \tilde{\Lambda}'(a;1)
  + \tilde{\Lambda}(a;1) \eta_1 (ak)
  \right) \ ,   \\
   & \mathrm{Res}_{z=1}
   \left(
   \zeta(z) \tilde{\Lambda}(a;z)
   \Phi(z;ak)(\log(h/a) -\eta(z;ak))
   \right)  \\
   & = x \Phi(1;ak) \tilde{\Lambda}(a;1)
   \left( \log(h/a)-\eta_1 (ak) \right) \ , 
   \nonumber
\end{split}
\end{equation}
it follows that
\begin{equation}
\begin{split}
  R_3 & =  x 
    \sum_{{\begin{substack}{a \mid h
         \\ (a,k)=1}\end{substack}}}
   \Phi(1;ak)( \tilde{\Lambda}(a;1) ( \log(x/e) + \log(h/a) )
   + \tilde{\Lambda}'(a;1)) \ . 
\end{split}
\end{equation}
By the identities $\Phi(1;ak)= \frac{\phi(ak)}{ak}=
\frac{\phi(k)}{k}(1-1/p)$,
$\tilde{\Lambda}(a;1)= \Lambda(a) \frac{p}{p-1}$,
$\tilde{\Lambda}'(a;1)= -\Lambda(a)\frac{p \log p}{(p-1)^2}$
we derive
\[
  R_3 = \frac{x\phi(k)}{k} 
   \sum_{{\begin{substack}{a \mid h
         \\ (a,k)=1}\end{substack}}}
 \Lambda(a) \left(
 \log(x/e) + \log(h/a) - \frac{\log p}{p-1}
 \right) \ . 
\]
Combining $R_2$ and $R_3$ we have 
\begin{align}
  R_2 + R_3 & =  \frac{x \phi(k)}{k} \left(
  \sum_{a \mid h} \Lambda(a) \left(
  \log(x/e)+ \log(h/a) \right)  
 -   \sum_{{\begin{substack}{a \mid h
         \\ (a,k)=1}\end{substack}}} 
         \frac{\Lambda(a) \log p}{p-1}
         \right)   \nonumber \\
         & = \frac{x \phi(k)}{k} \left(
         \log(x/e) \log h + (\Lambda*\log)(h) - g(h,k)
         \right) \ . 
         \label{eq:R2pR3}
\end{align}
Combining our expressions for $R_1$~(\ref{eq:R1}) and $R_2+R_3$~(\ref{eq:R2pR3}) we see that
\begin{equation}
\begin{split}
   \mathrm{residue} & = \frac{x \phi(k)}{k}
   \left( 
   \frac{1}{2} (\log x)^2 + 2\log(x/e) \log h + (\Lambda*\log)(h) 
   \right. \\
    & 
    +(C_{0} -\eta_1(k)) \log (x/e) 
    +a_1 \eta_1(k)-\eta_2(k)-g(h,k)
   \ . 
    \nonumber
\end{split}
\end{equation}
It suffices to compute the other error terms in~(\ref{eq:cint}).  We have the standard bounds
$
   \left|  (\zeta'/\zeta)(z)  \right| \ll (\log |z|) 
$,
$
   |\zeta^{(j)}(z)| \ll (\log |z|)^{j}
$
for $j=1,2$, $\mathrm{Re}(z) \ge 1- \frac{c'}{|\mathrm{Im}(z)|}$
and $|\mathrm{Im}(z)| \ge 3$
(see \cite[p.\ 146, p.\ 158]{Te}). Note that by 
our decomposition $A(z) =B_1(z)+B_2(z)+B_3(z)$ we have 
\begin{equation}
\begin{split}
   |B_1(z)| &  \ll  (\log U + \eta_{1/2}(k))j(k)
   (\log^2 U + \log U (\log h + \eta_{1/2}(k))  
    \ , 
   \nonumber
\end{split}
\end{equation}
\begin{equation}
\begin{split}
   |B_{2}(z)+B_{3}(z)| 
   & \ll  j(k)\sum_{a \mid h} \Lambda(a) j(a)
   (\log^2 U + \log U (\log h + \eta_{1/2}(ak)))  
   \nonumber
\end{split}
\end{equation}
and thus $|A(z)| \ll j(k) \Lo^3$.   It follows that the horizontal integrals
in~(\ref{eq:cint}) are bounded by 
\[
   \int_{\sigma_{0}(U)}^{c} |A(\sigma \pm i U)| x^{\sigma} \frac{d \sigma}{
   |\sigma \pm i U|}
   \ll \frac{x^{c} j(k) \Lo^3}{U} \ll \frac{j(k)x \Lo^3}{U} 
\]
and the leftmost vertical integral in~(\ref{eq:cint}) is bounded by 
\[
    x^{\sigma_{0}(U)}\int_{-U}^{U} 
    \frac{  |A(\sigma_{0}(U)+iu)| du}{|\sigma_{0}(U)+iu|}
    \ll x j(k) \Lo^3 \log(U) \exp \left(- \frac{c \log x}{\log(|U|+2)} 
    \right) \ . 
\]
If we choose $U= \exp(\beta \sqrt{\log x})$ for an appropriate $\beta>0$
then these last two error terms are $O(j(k)x \exp(-\beta' \sqrt{\log x}))$
for some $\beta' > 0$.  We finally deduce from~(\ref{eq:cint}) that
\[
      \sum_{{\begin{substack}{u \le x
         \\ (u,k)=1}\end{substack}}} (\Lambda*\log)(hu) =
         \mathrm{residue} + O \left(
         (\tau(h)+j(k))x \exp \left(-C_2 \sqrt{\log x} \right)
         \right) \ . 
\]
However, note that one can show 
$j(k) \ll \exp (o(\sqrt{\log k}))$ and hence the error term can 
be written as $O(\tau(h) \exp(-C_2 \sqrt{\log x}))$ for a smaller
$C_2$.  \\

We give a brief sketch how to adapt this argument for the
resonator case assuming the large zero-free region conjecture
for $\zeta(s)$.   Obviously, the residue term will remain unchanged. 
Instead in this case, we will move the contour further left to 
the line $\mathrm{Re}(s)= \sigma_1(U)$ where $\sigma_1(t)
= 1 - \frac{0.25 c_0}{\log \log (|t|+4)}$.  In this region, one can
establish that 
$
     \left|  (\zeta'/\zeta)(z)  \right| \ll (\log \log  |z|)$,
$     
   |\zeta^{(j)}(z)| \ll (\log \log |z|)^{j}
$ 
for $|z| \gg 1$.  These results may be proven exactly as in 
Lemma \ref{Lbounds} that follows.  We deduce 
\[
    x^{\sigma_{1}(U)}\int_{-U}^{U} 
    \frac{  |A(\sigma_{1}(U)+iu)| du}{|\sigma_{1}(U)+iu|}
    \ll x j(k) \Lo^3 \log(U) \exp
    \mbox{$ \left(- \frac{c \log x}{\log \log(|U|+2)} 
    \right) $}
\]
on the left edge of the contour.   Choosing $U = \exp(
\frac{\beta \log x}{\log \log x})$ for some $\beta >0$
yields the smaller error term.  
\end{proof}
We shall require a bound for $\delta$~(\ref{eq:delta}) that occurs in the decomposition~(\ref{eq:E}). 
\begin{lem} \label{delta}
For $d,k,q \in \mathbb{N}$, $\psi$ a primitive character modulo $q$ and $kq \ll T$ we have  
\[
     |\delta(q,kq,d,\psi)| \ll \frac{(d,k) \log \log T}{\phi(k) \phi(q)} \ . 
\]
Moreover, if $kq$ is squarefree then this bound may be replaced by 
$(d,k)/(\phi(k) \phi(q))$. 
\end{lem}
\begin{proof}
Now for any $a,b \in \mathbb{N}$ we have 
$
    \phi(ab)\theta((a,b)) = \phi(a)\phi(b)
$
where $\theta(n) = \prod_{p \mid n} (1-1/p)$.  However, one can show that
$\theta(n) \gg (\log \log |3n|)^{-1}$.  From these observations it 
follows that 
\[
   |\delta(q,kq,d,\psi)| \ll 
   \log \log (kq) \sum_{e \mid (d,k)}  \frac{\phi(e)}{\phi(kq)}
   \ll \frac{(d,k) \log \log T}{\phi(q) \phi(k)} \ . 
\]
The second stated bound is obtained by the same method.
\end{proof}

\begin{lem} \label{theta}
Let $h$ be a positive multiplicative function.  
Let $1 \le k,q \le M$.
We shall provide a bound for 
\[
     \theta(\sigma)= \sum_{d \mid kq} \frac{(d,k) h(d)}{d^{\sigma}}
     \ . 
\]
$(i)$ We first establish:
\begin{equation}
  \theta(\sigma) \ll \left\{ \begin{array}{ll}
                   (1*h)(k) || \mbox{$\frac{h(n)}{n}$}||_1 
                   & \mbox{if $\sigma=1$} \\
                 \sqrt{k} (1*h)(k) ||h||_{\infty}
                  T^{\epsilon} & \mbox{if $\sigma=1/2$}                  
                      \end{array}
          \right.   \ .
          \nonumber
\end{equation}
$(ii)$  We assume that $kq$ is squarefree, $h(p) \ll f(p)$ 
where $f$ is defined by~(\ref{eq:resc}), and $q \ge 
\eta :=\left( \frac{2.5 \Lo}{\log \Lo} \right)$.  Then we obtain
\begin{equation}
  \theta(\sigma) \ll \left\{ \begin{array}{ll}
                   (1*h)(k) 
                   \exp \left(
                   o(\sqrt{\Lo})
                   \right)
                   & \mbox{if $\sigma=1$} \\
                 \sqrt{k} (1*h)(k) T^{\epsilon} 
                 & \mbox{if $\sigma=1/2$}                  
                      \end{array}
          \right.   \ . 
          \nonumber
\end{equation}
\end{lem}
\begin{proof}
$(i)$ We put $g=(d,k)$, $d=gd_1$, and $k=dk_1$ so that 
\begin{equation}
\begin{split}
  & \theta(\sigma) = \sum_{g \mid k}
  g \sum_{d \mid kq, g=(d,k)} \frac{h(d)}{d^{\sigma}}
  \ll \sum_{g \mid k} g^{1-\sigma}
  \sum_{d_1 \mid q} \frac{h(gd_1)}{d_{1}^{\sigma}}  \\
  & \ll  k^{1-\sigma}  \sum_{g \mid k} h(g)
  \sum_{d_1 \mid q} \frac{h(d_1)}{d_1^{\sigma}}
  \ . 
  \nonumber
\end{split}
\end{equation}
If $\sigma=1$ then we have the bound $(1*h)(k)|| \frac{h(n)}{n}||_1$ and if $\sigma=\frac{1}{2}$ then we apply
$ \sum_{d_1 \mid q} \frac{h(d_1)}{\sqrt{d_1}} \ll 
||h||_{\infty} T^{\epsilon}$.  These bounds prove part $(i)$.   For part
$(ii)$ $kq$ is squarefree and thus $(k,q)=1$.  It follows that
\[
    \theta(\sigma) = \sum_{d \mid k} h(d) d^{1-\sigma}
    \sum_{e \mid q} h(e) e^{-\sigma}
    \le k^{1-\sigma} (1*h)(k)
    \sum_{e \mid q} h(e) e^{-\sigma} \ . 
\]
Since $q$ is squarefree and $h$ is multiplicative 
\begin{equation}
\begin{split}
   &  \log 
     \sum_{e \mid q} \frac{h(e)}{ e^{\sigma}}
   = \sum_{p \mid q} \log \left(1+ \frac{h(p)}{p^{\sigma}} \right)
   \ll \sum_{p \mid q} \frac{h(p)}{p^{\sigma}} 
 \ll \sum_{p \mid q} \frac{f(p)}{p^{\sigma}}  \ . 
  \nonumber
\end{split}
\end{equation}
Noting that $f$ is supported on those $n$ such that $n \ge L^2$
we obtain
\begin{equation}
\begin{split}
  &  \sum_{p \mid q, p > L^2} \frac{f(p)}{p}
   = \sum_{p \mid q, p > L^2} \frac{L}{p (\log p)}
   \ll \frac{L \omega(q)}{L^2  \log L} 
  \ll \frac{\log q}{L \log(L) \log \log q} 
  = o (\sqrt{\log M}) 
  \nonumber
\end{split}
\end{equation}
for $L = \sqrt{\log M \log \log M}$ and $q \ge \exp
\left( \frac{2.5 \Lo}{\log \Lo} \right)$. 
Now denote the prime divisors of $q$ as $r_1, \cdots , r_k$. 
Let $p_1, \cdots , p_k$ denote the first $k$ primes. We have 
that 
\[
    \sum_{p \mid q} \frac{f(p)}{p^{1/2}}
    \le  \sum_{p_{i}  \le 2 \log q} \frac{L}{(\log p) p^{1/2}}
    \ll   \frac{L (\log q)^{1/2}}{(\log \log q)^2} 
    = o(\log T) 
\]
for $\exp \left( \frac{2.5 \Lo}{\log \Lo} \right) \le q \le M$.  
\end{proof}

The final lemma in this section provides bounds for $L^{(k)}(s,\chi)$ 
$k=1,2$ and $\frac{L^{'}}{L}(s,\chi)$ for $s$ just to the left of 
$\mathrm{Re}(s)=1$ in the critical strip.  We have
\begin{lem} \label{Lbounds}
Suppose that $\chi$ is a primitive Dirichlet character modulo $q$. 
For $s=\sigma+it$ we put $\tau=|t|+4$. \\
(i) There exists a constant $c>0$ such that if $\mathrm{Re}(s) \ge 1 - \frac{c}{\log(q\tau)}$ then 
\[
   |L^{(k)}(s,\chi)| \ll \log(q \tau)^{k} \ , \
  \left| \frac{L^{'}}{L}(s,\chi) \right| 
  \ll \log(q \tau)  \ . 
\]
$(ii)$ Assume the large zero-free region conjecture.  If 
$\mathrm{Re}(s) \ge 1-\frac{c_0/4}{\log \log(q\tau)}$
then 
\[
    |L^{(k)}(s,\chi)| \ll  \log \log (q \tau)^{k}  \ , \
  \left| \frac{L^{'}}{L}(s,\chi) \right| \ll \log \log(q \tau)
   \ . 
\]
\end{lem}
\begin{proof}
Part $(i)$ is classical and 
and the proofs can be found in
\cite[pp.\ 331-343]{MV}. 
For part $(ii)$ we shall follow the argument for bounding $\zeta(s)$ 
presented in \cite{Te} pages 158-160.  
We put $\mathfrak{a}= (1-\chi(-1))/2$ and we suppose without 
loss of generality that $t > 0$.  Suppose that there exists $c_0$ such 
that if $\rho=\beta+i \gamma$ is a zero of $L(s,\chi)$ then 
\[
    \beta < 1 - \frac{c_0}{\log \log(q (|\gamma|+4))} \ .       
\]
In fact we can thus deduce that 
$
    \min_{\rho} \left(
    \frac{1}{\rho}+ \frac{1}{z-\rho}
    \right) \ge 0
$
for $z= a + ib$ where $b\ge b_0$ is sufficiently large 
and $a \ge 1-\frac{c_0/2}{\log \log qt}$
(This follows from the argument in \cite[pp.\ 158-159]{Te}).
We now let $s=\sigma+it$ be a fixed complex number with $t$ sufficiently 
large and $\sigma \ge 1-\frac{c_0/4}{\log \log qt}$.
We put $s_0 = 1+\eta+ i t$ with $\eta=\frac{c_0/4}{\log \log qt}$.  Suppose that $w$ is a complex number satisfying $|w| \le 4 \eta$. 
The point $s_0+w = \sigma'+i t'$ satisfies $t' \ge b_0$ and $\sigma'
\ge 1-\frac{c_0/2}{\log \log q t}$ and thus 
\begin{equation}
     \mathrm{Re} \left( 
   \frac{1}{\rho}+ \frac{1}{s_0+w-\rho} 
   \right) \ge 0 \ . 
   \label{eq:pos}
\end{equation}
Consider the function 
\[
    F(w) = \frac{L'(s_0,\chi)}{L(s_0,\chi)}- 
    \frac{L'(s_0+w,\chi)}{L(s_0+w,\chi)}
    \ . 
\]
By the explicit formula for $\frac{L^{'}}{L}(s,\chi)$ (see chapter 14 of 
\cite{Da}) 
\begin{equation}
\begin{split}
  F(w) & = \frac{1}{s_0+w-1}-\frac{1}{s_0-1}
  +\frac{1}{2} \left(
  \frac{\Gamma'}{\Gamma} \left(\frac{s_0+w}{2}+\frac{\mathfrak{a}}{2} \right)
  -\frac{\Gamma'}{\Gamma} \left(\frac{s_0}{2}+\frac{\mathfrak{a}}{2} \right)
  \right)  \\
  & - \sum_{\rho} \left( \frac{1}{\rho}+ \frac{1}{s_0+w-\rho} \right)
  + \sum_{\rho} \left( \frac{1}{\rho}+ \frac{1}{s_0-\rho} \right) \ . 
  \nonumber
\end{split}
\end{equation}
By (\ref{eq:pos}) and Stirling's formula it follows that
\[
  \mathrm{Re} (F(w)) \le  A\log \log q \tau + 
  \left|
  \mathrm{Re}   \sum_{\rho} \left( \frac{1}{\rho}+ \frac{1}{s_0-\rho} \right)
  \right| \ . 
\]
Now the sum is 
$\ll
     \sum_{\rho} \frac{1}{|\rho||s_0-\rho|}    
$
since $\mathrm{Re}(s_0) \ll 1$. 
Writing $\rho=\beta + i \gamma$ we divide the above sum into  intervals.  
Note that $s_0 = 1+\eta +i t$ with $t>0$.
\begin{equation}
\begin{split}
  &   I_1= [2t ,\infty) \ ,  \ 
    I_2 =[t+h,2t| \ , \ 
    I_3 = [t-h,t+h] \ , \ 
    I_4 = [1, t-h]  \ , \   \\
  &   I_5= [-1,1] \ , \ 
    I_6 = [-1,-t] \ , \
    I_7 = (-\infty, -t) \ . 
    \nonumber
\end{split}
\end{equation}
Moreover, we set for $j=1, \ldots , 7$
$ \sigma_{j} = \sum_{\gamma \in I_j}  (|\rho||s_0-\rho|)^{-1} \ . 
$ 
Before proceeding we note that
$
  |\rho| = \sqrt{\beta^2+\gamma^2} \ge  \max(|\gamma|,|\beta|)
$
and
\[
  |s_0-\rho| = \sqrt{(1+\eta-\beta)^2 + (\gamma-t)^2}
  \ge \max ( (1+\eta-\beta), |\gamma-\tau|) \ . 
\]
We define $N(t,\chi)$ to be the number of zeros of $L(s,\chi)$ in 
the box $-\frac{1}{2} \le \mathrm{Re}(s) \le \frac{3}{2}$ 
and $|\mathrm{Im}(s)| \le t$.  We shall employ the well-known
bound $N(t,\chi) \ll t \log(q(|t|+2))$.   We have
\begin{equation}
\begin{split}
   \sigma_1 = \sum_{\gamma \ge 2\tau} \frac{1}{\gamma |t-\gamma|}
   \ll \sum_{\gamma \ge 2t} \gamma^{-2} 
   \ll \frac{\log q t}{t} \ ,
   \nonumber
\end{split}
\end{equation}   
\[
   \sigma_2 \ll \sum_{t+h \le \gamma \le 2t} \frac{1}{\gamma(\gamma-t)}
   \ll \int_{t+h}^{2t} \frac{dN(u,\chi)}{u(u-t)}
   \ll \log(qt) \int_{t+h}^{2t} \frac{du}{u(u-t)}
   \ll \frac{\log(qt) \log(t/h)}{t} \ , 
\]
\begin{equation}
\begin{split}
   \sigma_3 = \sum_{t-h \le \gamma \le t+h}
   \frac{1}{\gamma (1+\eta-\beta)}
    & \ll \frac{\log \log q\tau}{c_0'} \frac{N(t+h,\chi)-N(t-h,\chi)}{t}  \\
   & \ll \frac{h \log (qt)  \log \log q\tau}{t}   \ , 
   \nonumber
\end{split}
\end{equation}
\begin{equation}
\begin{split}
   \sigma_4 = \sum_{1 \le \gamma \le t-h} \frac{1}{|\rho||s_0-\rho|}
  & \ll    \sum_{1 \le \gamma \le t-h} \frac{1}{\gamma (t-\gamma)}
   \ll \int_{1^{-}}^{t-h} \frac{dN(t,\chi)}{t(t-h)} \\
   & \ll (\log qt) \int_{1}^{t-h} \frac{du}{u (u-h)} 
 \ll \frac{\log qt}{t-h} \ .
 \nonumber
\end{split}
\end{equation}
\[
    \sigma_5 = \sum_{|\gamma| \le 1} \frac{1}{|\rho||s_{0}-\rho|}
    \ll t^{-1} \sum_{|\gamma| \le 1} |\rho|^{-1} 
    \ll t^{ -1} \beta_{\min}^{-1} N(1,\chi)
   \ , \
\]
where $\beta_{\min}$ is the smallest positive real zero of $L(s,\chi)$.
By the large zero-free region conjecture $\beta_{min} > \frac{c_0}{\log \log(4q)}$
and thus 
\[
   \sigma_5 \ll t^{-1} \log \log(q) \log(q) \ . 
\]
Similarly,
\[
   \sigma_6 = \sum_{-t \le \gamma \le -1} \frac{1}{|\rho||s_0-\rho|}
   \ll \sum_{-t \le \gamma \le -1} \frac{1}{|\gamma||t-\gamma|}
   \ll t^{-1} \sum_{-t \le \gamma \le -1} \frac{1}{|\gamma|}
   \ll \frac{\log qt}{t}  \ ,
\]
\[
   \sigma_7 = \sum_{\gamma \le -t} \frac{1}{|\rho||s_0-\rho|}
   \ll \sum_{|\gamma| \ge t} \frac{1}{|\gamma|^2}
   \ll \frac{\log q t}{t} \ . 
\]
Combining bounds and choosing $h=t/\log(q \tau)$ we derive 
\[
    \mathrm{Re}(F(w)) \ll A \log \log (q\tau) + \frac{h\log (q\tau) \log \log (q\tau)}{t} 
    \ll \log \log (q\tau)  \ . 
\]
Since we have $|s-s_0| \le 2 \eta$ it follows from the Borel-Caratheodory 
theorem
  that 
\[
  \left| \frac{L'}{L}(s,\chi) \right| \ll
  \log \log (q\tau) + \left| \frac{\zeta'}{\zeta}(s_0,\chi) \right|
  \ll \log \log (q\tau) \ . 
\]
Therefore
\[
    \log\left( 
    \frac{L(s,\chi)}{L(s_0,\chi)}
    \right) = \int_{s_0}^{s} \frac{L'}{L}(w,\chi) \, dw 
    \ll |s-s_0| \log \log  (q\tau)  \ll 1  \ . 
\]
 Now note that  $
    |\log L(s_0,\chi)| \le \log \zeta(1+\eta) 
    = \log_3  (q\tau) + O(1) 
 $
 and it follows that 
 $
     \log L(s,\chi) \ll \log_3  (q\tau) 
 $
 for $\mathrm{Re}(s) \ge 1- \frac{c_0}{\log \log q\tau}$. 
 Now writing $L(s,\chi) = \exp(\log L(s,\chi))$ we have 
 we have $
    |L(s,\chi)| 
    \le \exp (|\log L(s,\chi)|)
    \le \exp(\log_3  (q\tau) + O(1)) 
    \ll \log \log (q\tau)$. 
\end{proof}

\section{Evaluation of $\mathcal{M}_0$: Proof of Proposition \ref{M0}}

With the previous lemmas in hand we are now set to evaluate $\mathcal{M}_0$.
\begin{proof}
In~(\ref{eq:M01}) we set $l=(m,k)$, $m=lm_1$, and $k=lk_1$ to obtain
\[
   \mathcal{M}_0 = \sum_{l \le M} \sum_{k_1 \le M/l}
   \frac{y_{lk_1}}{lk_1} \frac{\mu(k_1)}{\phi(k_1)}
   \sum_{{\begin{substack}{lm_1 \le \frac{lk_1 T}{2 \pi}
         \\ (m_1,k_1)=1}\end{substack}}} a(m_1 l) \ . 
\]
Rewriting $k_1$ as $k$  
\begin{equation}
  \mathcal{M}_{0} =
  \sum_{lk \le M} \frac{y_{lk}\mu(k)}{lk\phi(k)}
  S \left(\frac{kT}{2 \pi} ;l,k \right)  
\end{equation} 
where
\[
  S \left( \frac{kT}{2 \pi};l,k \right) = \sum_{{\begin{substack}{m \le \frac{kT}{2 \pi}
         \\ (m,k)=1}\end{substack}}} a(ml)
   \; \mathrm{and} \;
  a(r) = \sum_{{\begin{substack}{uv=r
         \\ v \le M}\end{substack}}} (\Lambda*\log)(u) x_{v} \ .
\]
Note that by Lemma \ref{conv} we may decompose this as 
\[
    S \left(\frac{kT}{2 \pi};l,k \right) 
    = \sum_{gh=l} \sum_{{\begin{substack}{gv \le M
         \\ (v,kh)=1}\end{substack}}} x_{gv}
         \sum_{{\begin{substack}{uv \le  \frac{kT}{2 \pi}
         \\ (u,k)=1}\end{substack}}} (\Lambda*\log)(hu) \ . 
\]
We have by Lemma \ref{shu} 
that 
\begin{equation}
\begin{split}
     \sum_{{\begin{substack}{u \le \frac{kT}{2 \pi v}
         \\ (u,k)=1}\end{substack}}} & (\Lambda*\log)(hu)
    = \frac{T \phi(k)}{2 \pi v} ( X_1(h,k,v)   +X_2(h,k,v))
    + O \left(\frac{\tau(h) kT}{v} \exp (-C_2 \sqrt{\log T})  \right)
    \label{eq:shu2}
\end{split}
\end{equation}
since $ \log(\frac{kT}{v}) \asymp \log T$ and 
\begin{equation}
\begin{split}
   X_1(h,k,v) & =  \frac{1}{2} \mbox{$ \log \left(\frac{kT}{2 \pi v } \right)^2 $} 
   + 2 \mbox{$\log \left( \frac{kT}{2 \pi v e} \right)$} \log(h)
   + (\Lambda*\log)(h) \ ,  \\
   X_2(h,k,v) & =   
    (C_{0} - \eta_1(k))  \mbox{$\log \left( \frac{kT}{2 \pi v e}  \right)$}
    +a_1\eta_{1}(k) -\eta_2 (k) -g (h,k) \ .  
    \nonumber
\end{split}
\end{equation}   
We set $\mathcal{M}_0 = \mathcal{M}_0'+ \mathcal{M}_{0}''$ where
$\mathcal{M}_0'$ is the contribution in 
$\mathcal{M}_0$ arising from $X_1$ and $X_2$ in~(\ref{eq:shu2}) and 
$\mathcal{M}_{0}''$ denote contribution arising from the the error term 
in~(\ref{eq:shu2}).  First the error term is 
\begin{equation}
\begin{split}
 &  \mathcal{M}_{0}'' \ll T \exp(-C_2 \sqrt{\log T})
 \sum_{lk \le M} \frac{|y_{lk}|}{l \phi(k)}
  \sum_{gh=l}\tau(h) \sum_{v \le M/g}  \frac{|x_{gv}|}{v}
   \\
  & \ll T \exp(-C_2 \sqrt{\log T})
  \log \log M  ||  \mbox{$\frac{x_n}{n}$}||_1 
  \sum_{lk \le M} \frac{(\tau*|x|)(l)|y_{lk}|}{lk}  \\
  & \ll T  \exp(-C_3 \sqrt{\log T}) 
  || \mbox{$\frac{x_n}{n}$}||_1  
   ||\mbox{$ \frac{(\tau_3*|x|)(n)y_n}{n}$}||_1
  \ . 
  \nonumber
\end{split}
\end{equation}
We now deal with $\mathcal{M}_{0}'$:
\begin{equation}
\begin{split}
  \mathcal{M}_0' 
  & = \frac{T}{2 \pi}
  \sum_{lk \le M} \frac{y_{lk} \mu(k)}{lk} 
  \sum_{gh=l} \sum_{{\begin{substack}{gv \le M
         \\ (v,kh)=1}\end{substack}}} \frac{x_{gv}}{v}
    ( X_1(h,k,v)+X_2(h,k,v)) \\
    & = \frac{T}{2 \pi}
  \sum_{ghk \le M} \frac{y_{ghk} \mu(k)}{ghk} 
  \sum_{{\begin{substack}{gv \le M
         \\ (v,kh)=1}\end{substack}}} \frac{x_{gv}}{v}
    ( X_1(h,k,v)+X_2(h,k,v))  \ . 
    \nonumber
\end{split}
\end{equation}
By the variable change $hk=u$ we have 
\[
  \mathcal{M}_0'
   =  \frac{T}{2 \pi}
  \sum_{gu \le M} \frac{y_{gu}}{gu}  
  \sum_{{\begin{substack}{gv \le M
         \\ (v,u)=1}\end{substack}}} \frac{x_{gv}}{v}
   \sum_{hk=u} \mu(k) ( X_1(h,k,v)+X_2(h,k,v))  \ . 
\]
Next we will check that
\[
   \sum_{hk=u} \mu(k) X_1(h,k,v) =  -
   \mbox{$
    \frac{1}{2}\Lambda_{2}(u)
   + \log \left( \frac{T}{2  \pi v e} \right) \Lambda(u)
   + \frac{1}{2} \log \left( \frac{T}{2 \pi v} \right)^2 \delta(u)
   $}   \ . 
\]   
This follows immediately from the identities:
\begin{equation}
\begin{split}
    & \sum_{d \mid u} \mu(d) = \delta(u) :=  \left\{ \begin{array}{ll}
                  1 & \mbox{if $u=1$} \\
                  0 & \mbox{if $u > 1$} \\
                  \end{array} \right. \ , \\
  & \sum_{d \mid u} \mu(d) \log d = - \Lambda(u) \ , \\
  & \sum_{d \mid  u} \mu(d) (\log d)^{2}
   =- 2(\log u) \Lambda(u) + \Lambda_{2}(u) \ ,  \\
   & \sum_{de=u} \mu(d) (\log d)(\log e) = (\log u) \Lambda(u) 
   - \Lambda_{2}(u)  \ , 
   \\
   & \Lambda_2(u) = \Lambda(u) \log(u) + (\Lambda*\Lambda)(u)
   \ . 
      \nonumber
\end{split}
\end{equation}
Our next step is to compute $\sum_{hk=u} \mu(k)X_2(h,k,v)$. 
We recall that Lemma \ref{arith} gives us
\begin{equation}
\begin{split}
    & \phi_1(u) = \sum_{hk=u} \mu(k) \eta_1(k)  \ , \
   \phi_2(u)= \sum_{hk=u} \mu(k) \eta_2(k)   \ , \\
    & \phi_3(u)=\sum_{hk=u} \mu(k) g(h,k)  \ , 
    \  \phi_4(u)=\sum_{hk=u} \mu(k) \eta_1(k) \log(k) \ . 
\end{split}
\end{equation}
 It follows 
from the definitions~(\ref{eq:phij}) and Lemma \ref{arith} that
\begin{equation}
\begin{split}
   &\sum_{hk=u} \mu(k)X_2(h,k,v)
    =  \mbox{$ \log \left( \frac{T}{2 \pi v e} \right)$}
   (C_0 \delta(u)- \phi_1(u))\\
   & \quad -C_0 \Lambda(u)- \phi_4(u)
   +a_1 \phi_1(u)-\phi_2(u)- \phi_3(u) \ . 
\end{split}
\end{equation}
Combining these identities we arrive at 
\begin{equation}
\begin{split}
     &   \sum_{hk=u} \mu(k) ( X_1(h,k,v)+X_2(h,k,v))  
     =  - \mbox{$\frac{1}{2}$} \Lambda_{2}(u)
   +R_1 \left( \log(\mbox{$\frac{T}{2 \pi v}$})   \right) \Lambda(u)  \\
  & \quad + \mbox{$\frac{1}{2}$}
   R_2 \left(  \log(\mbox{$\frac{T}{2 \pi v}$})
   \right) \delta(u)   
   +\alpha_1(u)  \tilde{R}_1 \left( 
   \log(\mbox{$\frac{T}{2 \pi v}$})
   \right)   + \alpha_2(u) 
\end{split}
\end{equation}
where $R_1, R_2,\tilde{R}_1$ are monic polynomials of degrees $1,2,1$.
Note that $\alpha_1(u)=- \phi_1(u)$ and
$\alpha_1$ is supported on prime powers.  In fact,   
$\alpha_1(p^{\alpha})=  \frac{\log p}{p-1}$.   Also 
$\alpha_2(u) = a_1 \phi_1(u)-\phi_2(u)- \phi_3(u)  - \phi_4(u)$
and it is  supported on those integers $n$
with $\omega(n) \le 2$.   Moreover, we have  
\begin{equation}
\begin{split}
    \alpha_2(p^{\alpha})  & =
    - \frac{(\alpha+1) (\log p)^2}{p-1}
    - \frac{\log p}{p-1}  \left(a_1+ \frac{p}{p-1} \right)  \\
   \alpha_2 (p^{\alpha} q^{\beta}) & =
   - (\log p)(\log q) \left(
   \frac{1}{p-1} + \frac{1}{q-1} 
   \right) \ . 
\end{split}
\end{equation}
Therefore
\begin{equation}
\begin{split}
     \mathcal{M}'_{0}  & = \frac{T}{2 \pi} \sum_{gu \le M} \frac{y_{gu}}{gu}
  \sum_{{\begin{substack}{gv \le M
         \\ (v,u)=1}\end{substack}}} \frac{x_{gv}}{v} 
  \left( 
  -  \mbox{$\frac{1}{2}$} \Lambda_{2}(u)
   + R_1 \left(
    \log(\mbox{$\frac{T}{2 \pi v}$})  \right)
   \Lambda(u)
   \right.  \\
   & \left. \quad  +  \mbox{$\frac{1}{2}$} 
   R_2 \left( \log(\mbox{$\frac{T}{2 \pi v}$}) \right) \delta(u) 
  \right. 
   \left. + \tilde{R}_1 \left( 
   \log(\mbox{$\frac{T}{2 \pi v}$})
   \right)  \alpha_1 (u) + \alpha_2(u)
  \right)  \ .  \\
  \nonumber
\end{split}
\end{equation}
We define $H(M;u,v) = \sum_{g \le \min \left(\frac{M}{u}, \frac{M}{v} \right)}
\frac{y_{ug}x_{vg}}{g}$ and thus 
\begin{equation}
\begin{split}
   \mathcal{M}_{0} & =
   \frac{T}{2 \pi}  \sum_{{\begin{substack}{u,v \le M
         \\ (u,v)=1}\end{substack}}}
   \frac{ c'(u,v)H(M;u,v)}{av}
  + \frac{T}{4 \pi} \sum_{gv \le M} \frac{y_{g}x_{gv}}{gv} 
    R_2 \left( \mbox{$\log \frac{T}{2 \pi v}$} \right)  \\
    & + O \left( T  \exp(-C_3 \sqrt{\log T}) 
  ||  \mbox{$ \frac{x_n}{n}||_1 
  || \frac{(\tau_3*|x|)(n)y_n}{n} $}||_1
  \right) 
\end{split}
\end{equation}
where
\[
    c'(u,v) =  - \mbox{$\frac{1}{2}$} \Lambda_{2}(u)
   + R_1 \left(  \log(\mbox{$\frac{T}{2 \pi v}$})  \right) \Lambda(u)
    + \tilde{R}_1 \left( 
   \log(\mbox{$\frac{T}{2 \pi v}$})
   \right)  \alpha_1(u) + \alpha_2(u)  \ . 
\]
Now in the resonator case the error term $O(\frac{\tau(h)kT}{v} \exp(-C_2 \sqrt{\log T}))$ 
in~(\ref{eq:shu2}) above is just replaced by  $O(\frac{\tau(h)kT}{v} \exp(- \frac{C_2'\log T}{\log \log T}))$.  The argument then proceeds identically and yields the same formula as above for $\mathcal{M}_0$  except with the error term 
\[
       O \left( T  \exp \left( \mbox{$-\frac{C_3' \log T}{\log \log T}$} \right)
  \mbox{$ || \frac{x_n}{n}||_1 
  || \frac{(\tau_3*|x|)(n)y_n}{n}||_1 $}
  \right)  \ . 
\]
\end{proof}

\section{Bounding $\mathcal{E}$: Proof of Theorem \ref{Ebd}} 
In~(\ref{eq:E}) we invert summation order and replace the variables 
$k$ by $kq$ and $m$ by $md$ to obtain
\begin{equation}
\begin{split}
   \mathcal{E}
   & = \sum_{1 < q \le M} \ \chiq \tau (\overline{\psi})
   \sum_{k \le M/q} \frac{y_{kq}}{kq}
   \sum_{d \mid kq} \delta(q,kq,d,\psi)
   \sum_{m \le \frac{kqT}{2 \pi d}} a(md) \psi(m)  \\
  & = \sum_{k \le M} 
  \frac{\mathcal{N} 
  \left(
  \mbox{$ \frac{M}{k}$},  \mbox{$\frac{kT}{2 \pi}$} , k 
  \right)}{k}
\end{split}  
\end{equation}
where
\begin{equation}
   \mathcal{N}(\xi,z,k)
   = \sum_{2 \le q \le \xi}
   \frac{y_{kq}}{q} \chiq
   \tau(\overline{\psi})
   \sum_{d \mid kq} \delta(q,kq,d,\psi)
   \sum_{m \le qz/d} a(md) \psi(m) \ .
   \label{eq:N}
\end{equation}
In our analysis of $\mathcal{N}(\xi,z,k)$ we have to distinguish 
between the two cases for the coefficients $x_n,y_n$.    
We define
\[
       \eta= \left\{ \begin{array}{ll}
                    \Lo^{A}
                   & \mbox{{\it in the divisor case}} \\
                   \exp \left( \frac{2.5 \log T}{\log \log T} \right)
                 & \mbox{{\it in the resonator case}}                  
                      \end{array}
          \right.   
\]
for an arbitrary positive constant $A>0$.
We now estimate the sum $  \mathcal{N}(\xi,z,k)$ by dividing up the range
of $q$ into $2 \le q \le \eta$
and $\eta < q \le \xi \le M$.   The
case $2 \le q \le \eta$ 
is analogous to the Siegel-Walfisz theorem.  That is, 
we shall estimate directly the sum $\sum_{m \le qz/d} a(md) \psi(m)$ by 
the classical contour integral method invoking the zero-free region for
Dirichlet $L$-functions and Siegel's bound for the exceptional zero. 
The case $\eta < q \le \xi$ is analogous to Gallagher \cite{G} 
and Vaughan's  \cite{V} proofs of the Bombieri-Vinogradov theorem.  Here we shall employ an analytic form of the
large sieve inequality for Dirichlet characters.  
Thus we shall divide up $\mathcal{E}$ as:
\begin{equation}
\begin{split}
   \mathcal{E} & = 
   \sum_{k < M/\eta} \frac{\mathcal{N}(\eta, \frac{kT}{2 \pi},k)}{k}
   + \sum_{M/\eta < k \le M} 
   \frac{\mathcal{N}(\frac{M}{k}, \frac{kT}{2 \pi},k)}{k}  \\
  &  + \sum_{k \le M/\eta} 
   \frac{\mathcal{N}( \mbox{$\frac{M}{k}$},
   \mbox{$\frac{kT}{2 \pi},k)$}- 
   \mathcal{N}(\eta, \mbox{$\frac{kT}{2 \pi}$},k)}{k} \ . 
\end{split}
\end{equation}
We abbreviate this to $\mathcal{E} = \mathcal{E}_1 + \mathcal{E}_2
+ \mathcal{E}_3$.  Shortly we shall establish
\begin{prop} \label{propE12}
 $(i)$ If $x_n, y_n$ satisfy~(\ref{eq:divc}),~(\ref{eq:divc2}) then 
\begin{equation}
     \mathcal{E}_1 + \mathcal{E}_2 \ll  T \exp(-C_{4} \sqrt{\log T}) 
              \mbox{$ \left| \left| \frac{j(k)\tau_{r}(k) \tau_{r+3}(k)}{k}
              \right| \right|_1 $} \eta^{\frac{3}{2}+\epsilon} 
              \label{eq:E12div}
\end{equation}
for some $C_{4}=C_{4}(A)>0$ where $\eta=\Lo^A$.   \\

\noindent $(ii)$  Assume the large zero-free region conjecture.
If $x_n=y_n=f(n)$ then 
\begin{equation}
     \mathcal{E}_1 + \mathcal{E}_2 \ll 
    \exp 
     \mbox{$\left( 
    \frac{ (-c_0/8 + o(1))  \log T}{\log \log T} 
     \right) 
     \left|\left|
     \frac{j(k) f(k) (\tau_3*f)(k)}{k} \right|\right|_{1} $} \eta^{\frac{3}{2}}
     \ . 
     \label{eq:E12res}
\end{equation}
where $\eta=\exp(\frac{2.5 \log T}{\log \log T})$.
\end{prop}
We also show that 
\begin{prop}  \label{propE3}
$(i)$   If $x_n, y_n$ satisfy~(\ref{eq:divc}),~(\ref{eq:divc2}) then 
there exists a $C_5 >0$ such that 
\begin{equation}
 \mathcal{E}_3 
 \ll T\Lo^{C_5}   \mbox{$ || \frac{\tau_{r}(k)^2}{k}||_{1}^2 $}
 \eta^{-1/2}
  + T^{\frac{3}{4}+\frac{\theta}{2} + \epsilon}
  \label{eq:E3div}
\end{equation}
where $\eta = \Lo^A$. \\

\noindent $(ii)$   Assume the large zero-free region conjecture.
If $x_n=y_n=f(n)$ then 
\begin{equation}
  \mathcal{E}_3 \ll 
  T \Lo^3 
  \mbox{$
  || \frac{f(k)^2}{k}||_1 || \frac{ f(k) (\tau*f)(k)}{k}||_{1} $}
  \eta^{-\frac{1}{2}}+ 
  ||f||_{\infty}^2 \, 
  T^{\frac{3}{4}+ \frac{\theta}{2}+\epsilon}
  \ . 
  \label{eq:E3res}
\end{equation}
where $\eta =\exp(\frac{2.5 \log T}{\log \log T})$.
\end{prop}
With the above bounds for $\mathcal{E}_1, \mathcal{E}_2$, and $\mathcal{E}_3$ we deduce Theorem $\ref{Ebd}$ which provides
a bound for $\mathcal{E}$.
 
\begin{proof}[Proof of Theorem \ref{Ebd}] We begin with part $(i)$: the divisor case.  Since $\eta= \Lo^{A}$, $|| \mbox{$ \frac{\tau_{r}(k)^2}{k}$}||_{1},
|| \mbox{$ \frac{j(k)\tau_{r}(k)\tau_{r+3}(k)}{k}$}||_{1} \ll \Lo^{C'}$ 
it follows from Propositions
\ref{propE12} and \ref{propE3} that 
\[
   \mathcal{E} \ll T \exp(-C_{6} \sqrt{\log T})
   + T L^{C_7-0.5A} +  
   T^{\frac{3}{4}+\frac{\theta}{2} + \epsilon}
\]
for some $C_{6}=C_{6}(A),C_7 >0$.  Choosing $A=
2(A'+C_7)$ yields
\[
   \mathcal{E} \ll_{A'}   T (\log T)^{-A'}
    +   T^{\frac{3}{4}+\frac{\theta}{2} + \epsilon} \ . 
\]
We next prove part $(ii)$ of the theorem: the resonator case. 
Since $\eta= \exp( \frac{2.5 \Lo}{\log \Lo})$,
\[
   || \mbox{$\frac{f(k)^2}{k}$}||_1 \ , 
   \  || \mbox{$\frac{ j(k)f(k)(\tau_j*f)(k)}{k}$}||_1 \ll 
   \mbox{$ \exp 
   \left( \frac{(0.5+o(1)) \log T}{\log \log T} \right) $} \ , \  
   \nonumber
\]
and $||f||_{\infty} \ll T^{\theta}$ it follows that
\begin{equation}
\begin{split}
  \mathcal{E} 
  \ll T \left( \mbox{$ \exp \left(
  \frac{ (-c_0/8 +4.25+o(1)) \log T}{\log \log T}
  \right)
  + \exp \left(
  \frac{ (-0.25+o(1)) \log T}{\log \log T}
  \right) 
  $}
  \right) + T^{\frac{3}{4}+ \frac{3 \theta}{2} + \epsilon}
  \nonumber
\end{split}
\end{equation}
If $c_0$ is sufficiently large we
have established there exists a $C_{8} > 0$ such
that 
$
     \mathcal{E} \ll T   \exp \left(
     - \frac{C_{8} \log T}{\log \log T} 
     \right)   + T^{\frac{3}{4}+ \frac{3 \theta}{2} + \epsilon} 
$.
This completes the proof of Theorem \ref{Ebd}.   We have now reduced
the proof to establishing the bounds of Propositions \ref{propE12}
and \ref{propE3}. 
\end{proof}

\subsection{Bounding $\mathcal{E}_1, \mathcal{E}_2$: 
Proof of Proposition \ref{propE12}}  
 In this section we will bound $\mathcal{E}_1$ and $\mathcal{E}_2$
 and thus establish Proposition \ref{propE12}.  Note that $\mathcal{E}_1$, $\mathcal{E}_2$ each take the form
\begin{equation}
    \mathcal{E}_i = \sum_{k \le M} 
    \frac{\mathcal{N}(\xi, \mbox{$\frac{kT}{2 \pi}$},k)}{k}
    \label{eq:E12} \quad (  i=1,2)
\end{equation}
with $\xi \le \eta$.  
\begin{proof}[Proof of Proposition \ref{propE12}]
We shall evaluate 
$\mathcal{N}(\xi, \frac{kT}{2 \pi},k)$ by invoking the bound from
 Lemma \ref{delta} for  $\delta(q,k,d,\psi)$ and the following:
\begin{lem} \label{smallq}
Let $\psi$ be a non-principal character modulo $q$, $T \ll w \ll T^2$,
and $d \ll T$.  \\
{\it Divisor case}.  For any $A>0$ there exists a $C_{9} > 0$ such that
\[
   \sum_{m \le w} a(md) \psi(m)
   \ll_{A} j(d) (\tau*|x|)(d) w  \exp(- C_{9} \sqrt{\log T})
\]
for all $q \le \Lo^{A}$. \\
{\it Resonator case}.  Assume the large zero-free region conjecture
for $L(s,\psi)$.  Then we have
\[
     \sum_{m \le w} a(md) \psi(m)
   \ll   j(d) (\tau*|x|)(d) w  
   \mbox{$\exp \left(- \frac{(c_0/8) \log T}{\log \log T} \right)$} \ . 
\] 
for $q \le \exp(\frac{2.5 \Lo}{\log \Lo})$. 
\end{lem}  
The evaluation of  $\mathcal{N}(\xi,
\frac{kT}{2 \pi}, k)$  is split
in the two cases. 

We first consider part $(i)$.  That is, $x_n, y_n$ satisfy
(\ref{eq:divc}) and (\ref{eq:divc2}). 
By Lemma \ref{smallq} we obtain for $\xi \ll  \eta =\Lo^{A}$ 
\begin{equation}
\begin{split}
 &   \mathcal{N} \left(\xi, \mbox{$\frac{kT}{2 \pi}$},k \right)  \\
  & \ll kT \exp(-C_{9} \Lo^{1/2})
   \sum_{q \le \xi_k} |y_{kq}|
   \chiq |\tau(\overline{\psi})|
   \sum_{d \mid kq} |\delta(q,kq,d,\psi)| \frac{j(d)(\tau*x)(d)}{ d}
   \nonumber
\end{split}
\end{equation}
By Lemma \ref{delta}, $|\tau(\overline{\psi})| \le \sqrt{q}$,
$|x_n|,|y_n| \ll \Lo^{C} 
\tau_{r}(n)$ we see that $\mathcal{N}(\xi,\frac{kT}{2 \pi}, k)$ 
is bounded by 
\[
   \ll \frac{j(k)\tau_{r}(k) k T }{\phi(k)} 
  \exp(-C_{10} \sqrt{\Lo})    \sum_{q \le \xi} j(q)\tau_{r}(q) \sqrt{q}
    \sum_{d \mid kq} \frac{(d,k) (\tau*\tau_{r})(d)}{d}  \ . 
\]    
By Lemma \ref{theta}  and $\frac{k}{\phi(k)} \ll \log \Lo$
this is further bounded by 
\[
   \ll 
    j(k)\tau_{r+3}(k) \tau_{r}(k) T \exp(-C_{11} \sqrt{\Lo})    \sum_{q \le \eta} j(q)\tau_{r}(q) \sqrt{q} \ . 
\]
The last sum is $\ll \eta^{\frac{3}{2}+\epsilon}$ and thus
\[
   \mathcal{N} \left(\xi, \mbox{$\frac{kT}{2\pi}$},k \right)  \ll
    j(k)\tau_{r+3}(k) \tau_{r}(k)  T  \exp(-C_{11} \sqrt{\Lo})  
    \eta^{\frac{3}{2}+\epsilon}
\]
for some $C_{11} >0$. 
Therefore by~(\ref{eq:E12}) we have 
\[
   \mathcal{E}_1 + \mathcal{E}_2
   \ll  
   T \exp(-C_{11} \sqrt{\log T})
  \mbox{$ \left| \left| \frac{ j(k)\tau_{r+3}(k) \tau_{r}(k)}{k} \right|
   \right|_{1} $} \eta^{\frac{3}{2}+\epsilon} 
   \ . 
\]

We now establish part $(ii)$. 
Here we assume $x_n=y_n=f(n)$ and $\eta
= \exp(\frac{2.5 \Lo}{\log \Lo})$.  As before $\mathcal{N}(\xi, \frac{kT}{2 \pi},k)$ is bounded by 
\[
    kT \exp  \mbox{$ \left(-\frac{(c_0/8) \Lo}{\log \Lo} \right) $}
   \sum_{q \le \xi_k} f(kq)
   \chiq |\tau(\overline{\psi})|
   \sum_{d \mid kq} |\delta(q,kq,d,\psi)| \frac{j(d)(\tau*f)(d)}{ d} \ . 
\]
By Lemma \ref{delta} and $|\tau(\overline{\psi})| \le \sqrt{q}$,
we further bound this by 
\[
 \ll \frac{j(k)f(k) k T}{\phi(k)} 
  \exp  \mbox{$ \left(-\frac{ (c_0/8) \Lo}{\log \Lo} \right)   $} 
  \sum_{q \le \xi} j(q)f(q) \sqrt{q}
    \sum_{d \mid kq} \frac{(d,k) (\tau*f)(d)}{d}  \ . 
\]    
By Lemma \ref{theta} we obtain 
\[
     \ll 
    \frac{j(k)(\tau_{3}*f)(k) f(k) k T}{\phi(k)} 
    \exp \left(- \mbox{$\frac{(c_0/8)\Lo}{\log \Lo} + o(\sqrt{\Lo}) $} \right)   
    \sum_{q \le \eta} j(q)f(q) \sqrt{q} \ . 
\]
Obviously $\sum_{q \le \eta} j(q) f(q) \sqrt{q}
\ll \eta^{\frac{3}{2}} || \mbox{$ \frac{f(k)}{k}$}||_1$.  Thus 
\[
   \mathcal{N}(\xi, \mbox{$\frac{kT}{2 \pi}$},k)  \ll 
    j(k)(\tau_{3}*f)(k) f(k) T  
    \exp \mbox{$ \left(\frac{(-(c_0/8)+o(1))\Lo}{\log \Lo} \right) $}
    ||\mbox{$ \frac{ f(k)}{k}$}||_1 \eta^{\frac{3}{2}}    \ .
\]
Therefore by~(\ref{eq:E12})  
\begin{equation}
\begin{split}
  &  \mathcal{E}_1 + \mathcal{E}_2
  \ll  
   T \exp  \mbox{$ \left( \frac{(-c_0/8+o(1))\Lo}{\log \Lo}  \right)
   \left| \left| \frac{j(k)(\tau_{3}*f)(k) f(k)}{k}
   \right| \right|_{1} $}
   ||f(k)/k||_1 \eta^{\frac{3}{2}} \ . 
       \nonumber
\end{split}
\end{equation}
\end{proof}

\subsection{Proof of Lemma \ref{smallq}} 

We now establish Lemma \ref{smallq} which was central to establishing our 
bounds for $\mathcal{E}_1,\mathcal{E}_2$.  
\begin{proof}
Recall that $T \ll w \ll T^2$ and $d \ll T$. 
By Perron's formula, we have 
\[
    \sum_{m \le w} a(md) \psi(m)
    = \frac{1}{2 \pi i} \int_{\kappa -iU}^{\kappa+iU}
    A(s,\psi,d) w^s \frac{ds}{s} + 
    O( \epsilon )  
\]
where $ \kappa = 1+ O((\log w)^{-1})$
and 
\[
   \epsilon\ll \sum_{n \ne w} \left(\frac{w}{n} \right)^{\kappa}
   |a(dn)|
   \min \left(1 , U^{-1}|\log(w/n)|^{-1} \right)  + |a(w)| \ . 
\]
We will first show that $\epsilon$ is small.  Let $\epsilon= \epsilon_1+\epsilon_2$
where $\epsilon_1$ is the contribution from those terms with 
$n > 1.5w$ and $n< 0.5w$ and $\epsilon_2$ consists of the other terms. 
We observe that $(\Lambda*\log)(n) \le \log^2 n$ and hence
$
    |a(n)| \le (\log^2 n) (1*|x|)(n)$.
It follows that
$\epsilon_1$ is bounded by 
\[
   \frac{w}{U} 
   \sum_{n=1}^{\infty} 
   \frac{|a(dn)|}{n^{\kappa}}
   \ll \frac{w}{U}  (1*|x|)(d)
   \left(
   \Lo^2
   \sum_{n \ge 1} \frac{(1*|x|)(n)}{n^{\kappa}}
   +  \sum_{n \ge 1} \frac{\log^2 n (1*|x|)(n)}{n^{\kappa}}
   \right) \ . 
\]
Observe that 
\[
   \sum_{n =1}^{\infty} \frac{(1*|x|)(n)}{n^{\kappa}}
   = \sum_{a=1}^{\infty} \frac{1}{a^{\kappa}}
   \sum_{b \le M} \frac{|x_b|}{b^{\kappa}}
   \ll  \Lo  ||x_n/n||_{1}  \ . 
\]
A similar calculation gives
\[
    \sum_{n =1}^{\infty} \frac{(\log n)^2 (1*|x|)(n)}{n^{\kappa}}
    \ll  \Lo^{3} || \mbox{$ \frac{x_n}{n}$}||_{1} 
\]
and thus 
$
   \epsilon_{1}
   \ll  \frac{w}{U}  \Lo^{3} (1*|x|)(d)  ||\mbox{$\frac{x_n}{n}$}||_{1}
      \ . 
$
We now deal with $\epsilon_2$.  Since
$nd \ll T^{3}$ the contribution from those terms in 
$0.5w \le n < w$ is 
\begin{equation}
\begin{split}
  &  \sum_{0.5w < n < w} |a(nd)| \min(1, U^{-1}|\log(w/n)|^{-1})  \\
  &  \ll  (1*|x|)(d) \Lo^2
   \sum_{0.5w < n <  w}  (1*|x|)(n)
   \min(1,U^{-1}|\log(w/n)|^{-1}) \ . 
   \nonumber
\end{split}
\end{equation}
Since 
$
   |\log(w/n)|^{-1} \ll \frac{n}{|w-n|}
$ 
for $w/2 \le n < w$  it follows that the last sum is 
\[
  \sum_{w -\frac{w}{U} \le n < w} (1*|x|)(n)  +
  \frac{w}{U}\sum_{\frac{w}{2} < n < w-\frac{w}{U}} 
  \frac{(1*|x|)(n)}{w-n} \ . 
\] 
However by Lemma \ref{shiu} the first sum is  
$
  \ll U^{-1}  || \mbox{$ \frac{x_n}{n}$}||_{1} \ . 
$
In the second sum we divide it up in to intervals of the form 
$[w - (k+1)\frac{w}{U},w-k\frac{w}{U}]$ for $1 \le k \le K$
with $K \ll U$.  By another application of Lemma \ref{shiu} 
the second term is 
\begin{equation}
\begin{split}
  &  U^{-1} \sum_{k = 1}^{K} 
   \sum_{w - (k+1)\frac{w}{U} < n < w-k\frac{w}{U}}
   \frac{(1*|x|)(n)}{w-n}
   \ll w^{-1} \sum_{k \le K} \frac{1}{k}
   \sum_{w - (k+1)\frac{w}{U} < n < w-k\frac{w}{U}}
   (1*|x|)(n) \\
   & \ll w^{-1} \sum_{k \le K} \frac{1}{k} \frac{w}{U}
   || \mbox{$ \frac{x_n}{n}$}||_1  \ll 
   U^{-1}  \Lo || \mbox{$ \frac{x_n}{n}$}||_1 \ . 
   \nonumber
\end{split}
\end{equation}
Note that an identical argument applies to the range $w < n <1.5w$
and thus 
$
    \epsilon_2 \ll (1*|x|)(d) U^{-1} \Lo || \mbox{$ \frac{x_n}{n}$}||_{1}
$.  
In summary,   
\[
     \sum_{m \le w} a(md) \psi(m)
    = \frac{1}{2 \pi i} \int_{\kappa -iU}^{\kappa+iU}
    A(s,\psi,d) w^s \frac{ds}{s} + 
    O \left( \frac{(1*|x|)(d) || \mbox{$\frac{x_n}{n}$}||_{1} 
    w \Lo^{3}}{U} \right)
   \ . 
\]
To complete the proof we require a bound for
\begin{equation}
     I :=\frac{1}{2 \pi i} \int_{\kappa -iU}^{\kappa+iU}
    A(s,\psi,d) w^s \frac{ds}{s}  \ . 
    \label{eq:zfr}
\end{equation}
In order to achieve this we need some understanding of the generating
function $A(s,\psi,d)$.  We will
show that the generating function $A(s,\psi,d)$ can be computed 
explicitly in terms of $L(s,\psi)$.   With our knowledge of $A(s,\psi,d)$
in hand we shall deform our contour left into the zero-free region of $L(s,\psi)$
and then bound $A(s,\psi,d)$ on this contour.
Since $a(n)=(\Lambda*\log*x)(n)$,  Lemma \ref{conv}
yields
\begin{equation}
   A(s,\psi,d) =
   \sum_{d_1 d_2 d_3 =d} 
   A_1 (s,1,d_1)A_2(s,d_1,d_2)A_3(s,d_1d_2,d_3)
   \label{eq:product}
\end{equation}
where 
\begin{equation}
\begin{split}
  A_1(s,u,v) & =
   \sum_{{\begin{substack}{mv < M
         \\ (m,u)=1}\end{substack}}}
  \frac{\psi(m) x_{mv}}{m^s}  \ ,  \\
  A_2(s,u,v) & = \sum_{(m,u)=1} 
  \frac{\psi(m) \log(mv)}{m^s} \ , \\
  A_3(s,u,v) & =
  \sum_{(m,u)=1}
  \frac{\psi(m) \Lambda(mv)}{m^s} \ .  
  \nonumber
\end{split}
\end{equation}
A calculation demonstrates that
\begin{equation}
  A_2(s,u,v)
  = L^{'}(s,\psi) \Phi(s,\psi,u)-L(s,\psi) \Phi^{'}(s, \psi,u)
  + (\log v) L(s,\psi) \Phi(s,\psi,u)
  \nonumber
\end{equation}
where  $
    \Phi(s,\psi,u) = \prod_{p \mid u} (1-\psi(p)p^{-s})
    = \sum_{n  \mid u} \frac{\mu(n) \psi(n)}{n^s}$
and 
\begin{equation}
  A_3(s,u,v) = \left\{ \begin{array}{ll}
                  - \frac{L^{'}}{L}(s,\chi) -
                  \sum_{p \mid u} \frac{\chi(p) \log p}{p^s -\chi(p)}
                   & \mbox{if $v=1$} \\
                  \frac{\log p}{1-\chi(p)p^{-s}} & \mbox{if $v = p^{l}, (u,p)=1$} \\
                  \log p & \mbox{if $v = p^{l}, p \mid u$}  \\
                  0 & \mbox{else}
                  \end{array}
          \right.   \ .
          \nonumber
\end{equation}
With these expressions in hand we now analyze the behaviour of $A(s,\psi,d)$ to the right of the line $\mathrm{Re}(s)=1$.  For $\mathrm{Re}(s) \ge 1/2$, 
 $|\Phi(s,\psi,u)| \le j(d)$, $|\Phi^{'}(s,\psi,u)| \ll j(d) (\log d)$ and thus
 \[
    |A_{2}(s,u,v)| \ll j(d) 
    (|L^{'}(s,\psi)| + |L(s,\psi)| \Lo)   \ , 
\]
 \[
    |A_{3}(s,u,v)| \ll
    \mbox{$ | \frac{L'}{L}(s,\psi) |$} + \Lo  \ . 
\]
It follows from~(\ref{eq:product}) and these two last bounds that 
\begin{equation}
   |A(s,\psi,d)| \le j(d)
    \left(
   |L^{'}(s,\psi)| + |L(s,\psi)| \Lo
   \right)
   \left(
   \mbox{$| \frac{L^{'}}{L}(s,\psi) |$}
    +  \Lo
   \right) 
   \sum_{d_1 d_2 d_3=d} |A(s,1,d_1)|  
   \label{eq:Asd}
\end{equation}
Now since $|A(s,1,d_1)| \ll |x_{d_1}|
 ||  \frac{x_n}{n}||_{1} M^{1-\sigma}$ 
and by the bounds for $L^{(j)}(s,\psi)$, $\frac{L'}{L}(s,\psi)$ from 
Lemma \ref{Lbounds} we obtain
\begin{equation}
    |A(s,\psi,d)| \le j(d) (\tau*|x|)(d)  \Lo^3
    || \mbox{$ \frac{x_n}{n}$}||_{1} M^{1-\sigma} 
    \label{eq:Asd2}
\end{equation}
unconditionally for $\mathrm{Re}(s) \ge 1- \frac{c}{\log(q(|t|+4))}$
and assuming the large zero-free region conjecture it is true
for $\mathrm{Re}(s) \ge 1 - \frac{c_0/4}{\log \log(q(|t|+4))}$.
We are now prepared to bound $I$.   The argument is again 
split in two cases.  \\

\noindent {\bf Case 1.} {\it Divisor case}. It follows that $A(s,\psi,d)$ has a meromorphic continuation to all 
of $\mathbb{C}$.  For all non-principal characters, $A(s,\psi,d)$ 
has at most one simple pole in the region
$
    \{ s = \sigma +it   : \sigma \ge \sigma_{1}(t) = 1- \frac{c}{\log q(|t|+2)}
    \}
$
where $c>0$ is an absolute effective constant.  By Siegel's theorem, this pole,
if it exists is a real number $\beta$ that satisfies 
\begin{equation}
    1-\beta \gg_{\epsilon} q^{-\epsilon}
    \label{eq:siegbd}
\end{equation}
where the constant is ineffective.  We shall let $\Gamma_1$ denote
the contour $\sigma= \sigma_1 (U)$ and $|t| \le U$. 
By~(\ref{eq:Asd2})  and $||\frac{x_n}{n}||_1 \ll \Lo^{C'}$ we have
\begin{equation}
   |A(s,\psi,d)| \ll j(d) (\tau*|x|)(d) \Lo^{C''}  
   \label{eq:Abd}
\end{equation}
where $s = \sigma_{0}(U)+it$, $|t| \le U$, $|s-1| \gg \Lo^{-1},$ and
$|s-\beta| \gg \Lo^{-1}$.  It follows from Cauchy's theorem
that 
\begin{equation}
\begin{split}
  \sum_{m \le w} a(md) \psi(m) &  \ll 
   \int_{\Gamma_1}
   \mbox{$ \left| A(s,\psi,d) \frac{w^s}{s}\right| ds $} \\
  & + \mbox{$\left| \mathrm{res}_{s=\beta} A(s,\psi,d) \frac{w^s}{s}  \right|$}
  + \frac{w }{U} (\tau*|x|)(d) \Lo^{4} 
  || \mbox{$ \frac{x_n}{n}$}||_1 \ . 
\end{split}
\end{equation}
By~(\ref{eq:Abd}) 
\begin{equation}
\begin{split}
   \int_{\Gamma_1}
   A(s,\psi,d) \frac{w^s}{s} \, ds 
   & \ll j(d)(\tau*|x|)(d) \Lo^{C''}  
   w \exp  \mbox{$ \left( 
   \frac{-c\log w}{\log(q(U+2))}
   \right) $} \\
   & \ll j(d)(\tau*|x|)(d) \Lo^{C''}  
   w \exp(-C_{12} \sqrt{\log w})
   \nonumber
\end{split}
\end{equation} 
since $q \le \Lo^{A}$ and $U= \exp(C_{13} \sqrt{\log w})$.
To bound the residue at $s=\beta$ (the possible Siegel zero)
we invoke Siegel's ineffective bound~(\ref{eq:siegbd}) 
to obtain
\[
    w^{\beta} \le w \exp \mbox{$ \left( -
    \frac{C_{13} \log w}{q^{\epsilon}}
    \right)$}
    \le w \exp  \mbox{$ \left( -
    \frac{C_{13} \log w}{ \Lo^{\epsilon A}}
    \right)$} \ll w \exp \left( 
    - C_{14} \sqrt{\Lo}
    \right)
\]
if $\epsilon \le (2A)^{-1}$.   Thus 
\[
   \mathrm{res}_{s=\beta} A(\beta,\psi,d)
   \frac{w^s}{s} \ll j(d) (\tau*x)(d) \Lo^{C''} 
   w \exp(-C_{14} \sqrt{\Lo}) \ . 
\]
Collecting all estimates yields
\begin{equation}
   \sum_{m \le w} a(md) \psi(m) 
   \ll_{A}  j(d)(\tau*|x|)(d)  
  w \exp(-C_{15} \sqrt{\Lo}) \ . 
   \label{eq:siegel}
\end{equation}
where $C_{15} = C_{15}(A)>0$ and $q \le \Lo^{A}$.   \\

\noindent {\bf Case 2.} {\it Resonator case}.  In this case we assume that for each $q$ that every primitive Dirichlet $L$-function $L(s,\psi)$ is non-vanishing
in the region 
$
     \{ s = \sigma+it : \sigma \ge \sigma_{2}(t) = 1- 
    \frac{c_0/4}{\log \log (q(|t|+4))} 
    \}  
$.  We shall let $\Gamma_2$ denote the contour with $\sigma = \sigma_{2}(U)$
and $|t| \le U$.   
By~(\ref{eq:Asd2}) and  Lemma \ref{Lbounds}
\[
   |A(s,\psi,d)| \ll j(d) (\tau*|x|)(d) \Lo^{4}   ||\mbox{$ \frac{x_n}{n}$} ||_1
   \exp 
   \mbox{$ \left( \frac{c_0/4\log M}{\log \log M}  \right) $}
\]
where $s = \sigma_{2}(U)+it$, $|t| \le U$, 
$|s-1| \gg \Lo^{-1}$.  By Cauchy's theorem 
\[   \sum_{m \le w} a(md) \psi(m)  \ll 
    \int_{\Gamma_2}\left| A(s,\psi,d) \frac{w^s}{s}\right| ds  
  + \frac{w }{U} (\tau*|x|)(d) \Lo^{4} ||\mbox{$ \frac{x_n}{n}$} ||_1
   \ . 
\]
Hence
\begin{equation}
\begin{split}
   \int_{\Gamma_2}
   A(s,\psi,d) \frac{w^s}{s} \, ds 
   & \ll j(d)(\tau*|x|)(d) \Lo^5  ||\mbox{$ \frac{x_n}{n}$} ||_1
   w \exp  \mbox{$ \left( \frac{c_0}{4} \left( \frac{ \log M}{\log_2 M} -
   \frac{\log w}{\log_2(q(U+2))} \right)
   \right)  $} 
   \nonumber
\end{split}
\end{equation} 
We choose $U = \exp(\frac{C_{17} \log w}{\log \log w})$.  Since
$q \le \exp(\frac{2.5 \Lo}{\log \Lo})$, $T \ll w \ll T^2$ we have
\[
    \int_{\Gamma_2}
    A(s,\psi,d) \frac{w^s}{s} \, ds 
   \ll j(d)(\tau*|x|)(d)   ||\mbox{$ \frac{x_n}{n}$} ||_1
   w \exp  \mbox{$ \left(  \frac{c_0}{4}(\theta-1) (1+o(1))\frac{\Lo}{\log \Lo}
   \right)   $}
\] 
As $ ||\mbox{$ \frac{x_n}{n}$} ||_1
\ll \exp ( 
o(\sqrt{\Lo}))$ and $0 < \theta < \frac{1}{2}$ we deduce 
\[
   \sum_{m \le w} a(md) \psi(m) 
   \ll  j(d)(\tau*|x|)(d)  w 
   \mbox{$ \exp \left( -\frac{(c_0/8)\log T}{\log \log T} \right) $}
\]
for $q \le \exp \left( \frac{2.5 \Lo}{\log \Lo} \right)$. 
\end{proof}

\subsection{Bounding $\mathcal{E}_3$:  Proof of Proposition \ref{propE3}} 
We now prove the bound for
\begin{equation}
    \mathcal{E}_3 = \sum_{k \le M/\eta} 
   \frac{\mathcal{N}(\frac{M}{k},\frac{kT}{2 \pi},k)- \mathcal{N}(\eta, \frac{kT}{2 \pi},k)}{k} \ . 
    \label{eq:E3def}
\end{equation}
\begin{proof}[Proof of Proposition \ref{propE3}]
 By Perron's formula applied with $U=T^{20}$, $T \ll w \ll T^2$,
and $\kappa =1 + (\log w)^{-1}$ we have 
\[
   \sum_{m \le w} a(md) \psi(m)
   = \frac{1}{2 \pi i} \int_{\kappa-iU}^{\kappa+iU}
   A(s,\psi,d)w^{s} \frac{ds}{s} + O(w^{\epsilon})  \ . 
\]
Combining this expression with the definition~(\ref{eq:N}) of $\mathcal{N}$
we obtain
\begin{equation}
\begin{split}
 & | \mathcal{N}( \mbox{$\frac{M}{k}$}, \mbox{$\frac{kT}{2\pi}$},k)
 - \mathcal{N}(\eta,\mbox{$\frac{kT}{2 \pi}$},k)|  \ll 
  \sum_{\eta \le q \le \frac{M}{k}}
   \frac{|y_{kq}|}{q} \\
 & 
   \cdot  \chiq |\tau(\overline{\psi})|
   \sum_{d \mid kq} |\delta(q,kq,d,\psi)|
    \left( \left|
   \int_{\kappa-iU}^{\kappa+iU} A(s,\psi,d) \left( \frac{qkT}{2 \pi d} \right)^{s} 
   \frac{ds}{s}
   \right| + O(T^{\epsilon}) \right)  
   \label{eq:N1bounds}
\end{split}
\end{equation}
By Lemma \ref{delta} the term containing $O(T^{\epsilon})$ contributes
\begin{equation}
\begin{split}
   & \ll T^{\epsilon}    \sum_{\eta \le q \le \frac{M}{k}} \frac{|y_k||y_q|}{\sqrt{q}}
   \chiq  \sum_{d \mid kq}
    \frac{(d,k)}{\phi(k) \phi(q)} 
    \ll T^{\epsilon}   \sum_{\eta \le q \le \frac{M}{k}} \frac{|y_k||y_q| \sigma_1(kq)}{ \phi(k) \sqrt{q}}  \ll  |y_k| T^{\epsilon} M^{\frac{3}{2}} \ . 
   \nonumber
\end{split}
\end{equation}
The first sum in~(\ref{eq:N1bounds}) is bounded by 
\begin{equation}
\begin{split}
   & \ll  \frac{|y_k| (\log \Lo) }{\phi(k)}
   \sum_{\eta \le q \le \frac{M}{k}} \frac{|y_q| \sqrt{q}}{q \phi(q)}
    \ \chiq
   \sum_{d \mid kq} \mu^2(kq)(d,k)
   \left|
   \int_{\kappa-iU}^{\kappa+iU} A(s,\psi,d) \left( \frac{qkT}{2 \pi d} \right)^{s} 
   \frac{ds}{s}
   \right|  \\
   & \ll  \frac{|y_k| (\log \Lo) }{\phi(k)}
 \max_{Q \le \frac{M}{k}}  \left( \mu^{2}(kQ)  \sum_{d \mid kQ} (d,k)
 \int_{\eta}^{\frac{M}{k}} z^{-1} dS(z) \right)
 \nonumber
\end{split}
\end{equation}
where  
\[
   S(z) = \sum_{q \le z} \frac{q^{1/2} |y_q|}{\phi(q)}
   \chiq
   \left|
   \int_{\kappa-iU}^{\kappa+iU} A(s,\psi,d) \left( \frac{qkT}{2 \pi d} \right)^{s} 
   \frac{ds}{s}
   \right|   \ . 
\]  
The next step is to dissect $A(s,\psi,d) = \sum_{m=1}^{\infty} a(md)\psi(m) m^{-s}$ via Vaughan's identity. 
We define the partial sum of $A(s,\psi,d)$
\[
    F = F(s,\psi,d) = \sum_{m \le  u} \frac{a(md)\psi(m)}{m^s} \ , 
\]
and the partial sum $G(s,\psi)$ of $L(s,\psi)^{-1}$ by 
\[
    G = G(s,\psi) = \sum_{m \le v} \frac{\mu(m) \psi(m)}{m^s} \ . 
\]
We choose the parameters
\[
    u=z^2 \ \mathrm{and} \ v =T^{1/2}
\]
where $z$ is a real variable satisfying
$\eta \le z \le \frac{M}{k}$.  
Vaughan's identity is 
\[
    A= (A-F)(1-LG)+ (F-FLG+ALG)  \ . 
\]
We write this as $A=H+I$ where
\[
   H=(A-F)(1-LG) \ \mathrm{and} \
   I = F-FLG+ALG \ . 
\]
It follows that 
\[
   \frac{1}{2 \pi i} 
   \int_{\kappa-iU}^{\kappa+iU}
   A(s,\psi,d)  \left(\frac{qkT}{2 \pi d} \right)^s \frac{ds}{s}
   = 
    \frac{1}{2 \pi i} 
   \int_{\kappa-iU}^{\kappa+iU}
   (H+I)(s,\psi,d)  \left(\frac{qk T}{2 \pi d} \right)^s \frac{ds}{s}  \ . 
\]
By the argument of \cite[p.\ 514]{CGG3} we have
\[
   \int_{\kappa-iU}^{\kappa+iU}
   I(s,\psi,d) \left(
   \frac{qkT}{2 \pi d}
   \right)^s \frac{ds}{s}
   = 
   \int_{\frac{1}{2}-iU}^{\frac{1}{2}+iU}
   I(s,\psi,d) \left(
   \frac{qkT}{2 \pi d}
   \right)^s \frac{ds}{s} + O(T^{-1}) \ . 
\]
Next we define
\[
   \mathcal{H}(z) = \sum_{q \le z} 
   \frac{|y_q| q^{3/2}}{\phi(q)} 
   \ \chiq \int_{-U}^{U} 
   |H(\kappa+it)| \frac{dt}{\kappa+|t|}
\]
where $\kappa =1 + O(\Lo^{-1})$ and 
\[
    \mathcal{I}(z) = \sum_{q \le z} 
   \frac{|y_q| q}{\phi(q)} 
   \ \chiq \int_{-U}^{U} 
   |I(1/2+it)| \frac{dt}{1/2+|t|} \ . 
\]
With these definitions in hand we obtain 
\begin{equation}
    |\mathcal{N} (\mbox{$\frac{M}{k}$},
    \mbox{$\frac{kT}{2 \pi}$} ,k )
    -\mathcal{N} (\eta,\mbox{$\frac{kT}{2 \pi}$} ,k )|
    \ll \sigma_1 + \sigma_2 + |y_k| M^{\frac{3}{2}} T^{\epsilon}
    \label{eq:Nbd}
\end{equation}
where
\begin{equation}
\begin{split}
    \sigma_1 & =
   \frac{(Tk) |y_k| \log \Lo}{\phi(k)} \max_{Q \le \frac{M}{k}} \mu^{2}(kQ)
   \sum_{d \mid kQ} \frac{(d,k)}{d} \int_{\eta}^{\frac{M}{k}}
    z^{-1} d \mathcal{H}(z) \ ,  \\
    \sigma_2 & =    
    \frac{(Tk)^{1/2} |y_k| \log \Lo}{\phi(k)} \max_{Q \le \frac{M}{k}} \mu^{2}(kQ)
   \sum_{d \mid kQ} \frac{(d,k)}{d^{1/2}} \int_{\eta}^{\frac{M}{k}}
   z^{-1} d \mathcal{I}(z)  \ . 
   \nonumber
\end{split}
\end{equation}
Next we will show the bounds 
\begin{equation}
     \int_{\eta}^{\frac{M}{k}}
    z^{-1} d \mathcal{H}(z)   \ll (1*|x|)(d) \Lo^{5} 
    ||  \mbox{$ \frac{x_k^2}{ k}$}||_{1}^{1/2}
     ||  \mbox{$ \frac{x_k}{ k}$}||_{1}
   \left( 
   \eta^{-1/2} + T^{-\delta}k^{-1/2}
   \right)  \ , 
    \label{eq:intH}
\end{equation}
\begin{equation}
   \int_{\eta}^{\frac{M}{k}}
    z^{-1} d \mathcal{I}(z)  \ll
    j(d) \tau_3(d)|x(d)|^2
   \left(
   T^{\theta+\epsilon} k^{-\frac{1}{2}}
   + T^{\frac{\theta}{2}+\frac{1}{4}+\epsilon}
   \right)
   \label{eq:intI}
\end{equation}
where $\delta =\frac{1}{4}-\frac{\theta}{2}$. We
deduce 
\begin{equation}
\begin{split}
  \sigma_1 & \ll T \Lo^{7} 
    || \mbox{$ \frac{x_k^2}{ k}$}||_{1}^{1/2}
     ||  \mbox{$ \frac{x_k}{ k}$}||_{1}
   \left( 
   \eta^{-1/2} + T^{-\delta} k^{-1/2}
   \right) 
    |y_k| \max_{Q \le \frac{M}{k}} \mu^{2}(kQ)
   \sum_{d \mid kQ} \frac{(d,k)(1*|x|)(d)}{d} \ , 
\end{split}
\end{equation}
\begin{equation}
  \sigma_2 \ll  \frac{|y_k|}{k^{\frac{1}{2}}}
  \max_{Q \le \frac{M}{k}} \mu^{2}(kQ)
  \sum_{d \mid kQ} \frac{(d,k)j(d)\tau_3(d) |x(d)|^2}{d^{1/2}}
  \left(
  T^{\frac{1}{2}+\theta+\epsilon} k^{-\frac{1}{2}}
  + T^{\frac{3}{4}+\frac{\theta}{2}+\epsilon}
  \right) \ . 
  \label{eq:sig2}
\end{equation}
We now bound  $ |\mathcal{N}(\frac{M}{k},\frac{kT}{2 \pi},k)-\mathcal{N}(\eta,\frac{kT}{2 \pi},k)|$ 
in the two cases: \\

\noindent {\bf Case 1}. {\it Divisor case}.  We have by Lemma \ref{theta} $(i)$
\begin{equation}
\begin{split}
   \sigma_1 & \ll T \Lo^{C_{18}} 
   \left( 
   \eta^{-1/2} + T^{-\delta}k^{-1/2}
   \right) 
    \tau_{r}(k)\max_{Q \le \frac{M}{k}} \mu^{2}(kQ)
   \sum_{d \mid kQ} \frac{(d,k)\tau_{r+1}(d)}{d}  \\
   & \ll T \Lo^{C_{19}} 
   \left( 
   \eta^{-1/2} + T^{-\delta}k^{-1/2}
   \right) 
    \tau_{r}(k) \tau_{r+2}(k) \ . 
    \nonumber
\end{split}
\end{equation}
By  Lemma  \ref{theta} $(i)$ and $|x_n|, |y_n| \ll T^{\epsilon}$
we obtain 
\[  \sigma_{2}   
 \ll   
  T^{\frac{1}{2}+\theta+\epsilon} k^{-1/2} + T^{\frac{3}{4}+\frac{\theta}{2}+\epsilon} \ . 
\]
From~(\ref{eq:E3def}),~(\ref{eq:Nbd}) and our bounds for $\sigma_i$
we have 
\begin{equation}
\begin{split}
  \mathcal{E}_3 & \ll  T \Lo^{C_{19}} \sum_{k \le M/\eta} 
   \frac{\tau_{r}(k) \tau_{r+2}(k)}{k} 
     \left( 
   \eta^{-1/2} + T^{-\delta}k^{-1/2}
   \right) 
   + 
  T^{\frac{1}{2}+\theta+\epsilon} +      
  T^{\frac{3}{4}+\frac{\theta}{2}+\epsilon}
   \\
  & \ll T \Lo^{C_{20}} \eta^{-\frac{1}{2}}
  + T^{\frac{3}{4}+ \frac{\theta}{2}+ \epsilon}
  \nonumber
\end{split}
\end{equation}
for $0 < \theta < 1/2$ as claimed. \\

\noindent {\bf Case 2}. {\it Resonator case}.  
By~(\ref{eq:intH}) and~(\ref{eq:f2n}) we have
\[
    \sigma_1 \ll T\exp  \mbox{$ \left(
  \frac{(0.25+o(1))\Lo}{\log \Lo}
  \right) $}
   \left( 
   \eta^{-1/2} + T^{-\delta}k^{-1/2}
   \right) 
   f(k) \max_{Q \le \frac{M}{k}} \mu^{2}(kQ)
   \sum_{d \mid kQ} \frac{(d,k)(1*f)(d)}{d} 
\]
By an application of Lemma \ref{theta} $(ii)$ this is
\[
  \sigma_1  \ll  T\exp 
  \mbox{$ \left(
  \frac{(0.25+o(1))\Lo}{\log \Lo} 
  \right) $}
   \left( 
   \eta^{-1/2} + T^{-\delta}k^{-1/2}
   \right) 
    f(k) (\tau*f)(k)  \ . 
\]
By ~(\ref{eq:sig2}), Lemma \ref{theta} $(ii)$ and $f(d) ,  (\tau*f)(d) \ll ||f||_{\infty}
 T^{\epsilon}$ (see~(\ref{eq:finf})) we find
\[
     \sigma_2 \ll  f(k) 
     ||f||_{\infty}^{2}
  \left(
  T^{\frac{1}{2}+\theta+\epsilon} k^{-1/2} + T^{\frac{3}{4}+\frac{\theta}{2}+\epsilon}
  \right)
\]
By~(\ref{eq:E3def}),~(\ref{eq:Nbd}) and our expressions for the $\sigma_i$
\begin{equation}
\begin{split}
 \mathcal{E}_3 & \ll
 T \exp \mbox{$ \left( \frac{(0.25+o(1)) \Lo}{\log \Lo} 
 \right) $}
  || \mbox{$ \frac{f(k) (\tau*f)(k)}{k}$}||_1 
   \eta^{-1/2} +  ||f||_{\infty}^2 (T^{\frac{1}{2}+\theta+\epsilon} +  T^{\frac{3}{4}+\frac{\theta}{2}+\epsilon})
   \nonumber
\end{split}
\end{equation}
since   $|| \mbox{$ \frac{f(k)}{ k^{3/2}}$}||_1 \ll 1$
and $|| \frac{f(k)}{k}||_1 \ll M^{\epsilon}$.
\end{proof}
\section{Establishing (\ref{eq:intH})}  
The argument of Proposition \ref{propE3} has been reduced to 
establishing (\ref{eq:intH}) and (\ref{eq:intI}).  In this
section we establish (\ref{eq:intH}).
We require the large sieve inequality:
\begin{equation}
  \sum_{q \le z} \frac{q}{\phi(q)}
  \chiq \int_{-U}^{U}
  \left| 
  \sum_{n} a_n \psi(n) n^{-it}
  \right|^2  \frac{dt}{\kappa+|t|}
  \ll \sum_{n} (n +z^2 (\log U)) |a_n|^2 \ . 
  \label{eq:lsi}
\end{equation}
In addition, we define for an arbitrary function
$\phi(s,\psi)$ the operator
\[
   \mathcal{A}(\phi)=
   \sum_{q \le z} \frac{q}{\phi(q)}
   \sideset{}{^*}\sum_{\psi \, \mathrm{mod} \, q}
   \int_{-U}^{U} |\phi(1/2+it,\psi)| \frac{dt}{\kappa+|t|} \ . 
\]
Notice that if $c$ is a constant then $\mathcal{A}(c \phi)=|c| \mathcal{A}(\phi)$
and also for two functions $\phi_{i}=\phi_{i}(s,\psi)$ for $i=1,2$ we have
$
     \mathcal{A}(\phi_1 \phi_2) \le \mathcal{A}(\phi_{1}^2)^{1/2}
     \mathcal{A}(\phi_{2}^2)^{1/2} 
$.
Recall that $H=(A-F)(1-LG)$.  By~(\ref{eq:lsi})
\begin{equation}
\begin{split}
   \mathcal{A}( (A-F)^2)  & = 
   \sum_{q \le z} \frac{q}{\phi(q)}  \sideset{}{^*}\sum_{\psi \, \mathrm{mod} \, q}
   |A(\kappa+it)-F(\kappa+it)|^2
   \frac{dt}{\kappa+|t|}  \\
  &  \ll \sum_{m \ge u} (m+z^2 \Lo) |a(md)|^2 m^{-2 \kappa}
  \ . 
\end{split}
\end{equation}
Since $|a(md)| \ll (\log^2 m) \Lo^2 (1*|x|)(d) (1*|x|)(m)$ 
for $d \le T$ 
\begin{equation}
\begin{split}
   &  \mathcal{A}( (A-F)^2)\ll (1*|x|)^2 (d) \Lo^4 \cdot 
  \sum_{m \ge u} \frac{(\log^4m )(1*|x|)^2(m)}{m^{2\kappa}}
  (m+z^2 \Lo) 
  \ . 
  \nonumber
\end{split}
\end{equation}
Since $2\kappa-1 = 1+ O(\Lo^{-1})$ we have
\begin{equation}
\begin{split}
   & \sum_{m \ge u} \frac{(\log^4m )(1*|x|)^2(m)}{m^{2\kappa-1}}
   = \sum_{v_1,v_2 \le M} |x_{v_1}||x_{v_2}|
    \sum_{{\begin{substack}{m \ge u
         \\ [v_1,v_2] \mid m }\end{substack}}}
  \frac{ \log^4 m}{m^{2 \kappa-1}}  \\
 & =  \sum_{v_1,v_2 \le M} 
  \frac{|x_{v_1}||x_{v_2}|}{[v_1,v_2]^{2 \kappa-1}}
  \sum_{m' \ge \frac{u}{[v_1,v_2]}} 
  \frac{\log^4 (m' (v_1,v_2))}{(m')^{2\kappa-1}}  
  \ll \Lo^5 \sum_{v_1,v_2 \le M} 
  \frac{|x_{v_1}||x_{v_2}|}{[v_1,v_2]}  
  \nonumber
  \end{split}
\end{equation}
Since $[v_1,v_2]^{-1}= (v_1 v_2)^{-1} \sum_{g \mid (v_1,v_2)}
\phi(g)$ the last expression is
\[
   \ll \Lo^5 \sum_{g \le M} \frac{\phi(g) |x_g|^2}{g^2}
   (
   \sum_{{\begin{substack}{v \le \frac{M}{g}
         \\ (v,g)=1}\end{substack}}}
   \frac{|x_v|}{v}
   )^2 
   \ll \Lo^5 
   \mbox{$  ||\frac{x_n^2}{n} ||_1 || \frac{x_n}{n} ||_{1}^{2} $}\ . 
\]
An analogous calculation establishes
\[
   \sum_{m > u} \frac{(\log m)^4 (1*|x|)(m)^2}{m^{2 \kappa}}
   \ll \Lo^{4} || \mbox{$ \frac{(1*|x|)^2(k)}{k}$}||_1 u^{-1}
   \ll \Lo^5 \mbox{$  ||\frac{x_n^2}{n} ||_1 || \frac{x_n}{n} ||_{1}^{2} $}
   u^{-1} \ . 
\]
Therefore 
$
    \mathcal{A}( (A-F)^2) \ll (1*|x|)(d)^2 \Lo^6
     || \mbox{$\frac{x_{n}^2}{ n}$}||_{1} 
      || \mbox{$\frac{x_{n}}{ n}$}||_{1}^{2} 
$. 
Moreover,  in \cite{CGG3} it is established that
$
  \mathcal{A}((1-LG)^2)
     \ll (1+z^2 v^{-1}) \Lo^{4}$.
Thus we obtain
\begin{equation}
\begin{split}
  \mathcal{H}(z) 
   & \le \sqrt{z} \mathcal{A}(|(A-F)(1-LG)|)  \\
   & \ll (1*|x|)(d) \Lo^{5}  
   || \mbox{$ \frac{x_k^2}{ k}$}||_{1}^{1/2}
    || \mbox{$ \frac{x_k}{ k} $}||
   \sqrt{z} (1+ zT^{-1/4})  
   \nonumber  
\end{split}
\end{equation}
where in the last line we applied Cauchy-Schwarz. 
Since $M=T^{\theta}$
 we deduce  
\[
   \int_{\eta}^{\frac{M}{k}} z^{-1} d \mathcal{H}(z)
   \ll (1*|x|)(d) \Lo^{5} || \mbox{$ \frac{x_k^2}{ k}$}||_{1}^{1/2}
    || \mbox{$ \frac{x_k}{ k}$}||_{1}
   \left( 
   \eta^{-1/2} +  T^{\frac{\theta}{2}-\frac{1}{4}} k^{-1/2}
   \right)  \ . 
\]

\section{Establishing ~(\ref{eq:intI}) }
In this section, we are not so precise about bounds. 
This is since we will have a small power savings 
from the main term. 
We set $s=1/2+it$ and we now provide a bound for $\mathcal{I}(z)$. 
Since
$
   \mathcal{I}(z) 
   = \mathcal{A}(F-FLG+ALG)  
$
we see that
\begin{equation}
  \mathcal{I}(z)
  \ll  \mathcal{A}(F^2)^{1/2} 
    \mathcal{A}(1)^{1/2} +
    \mathcal{A}(F^2)^{1/2}
     \mathcal{A}(L^4)^{1/4} 
     \mathcal{A}(G^4)^{1/4} +
     \mathcal{A}( |LA| |G|) \ . 
     \label{eq:Ix}
\end{equation}
It follows from~(\ref{eq:Asd}) that
\[
 |A(s,\psi,d)L(s,\psi)| \le  j(d) \tau_3(d)T^{\epsilon}
  (|L(s,\psi)|^2 + |L^{'}(s,\psi)|^2)
 | \mathcal{B}(s)|
\]
for some Dirichlet polynomial $\mathcal{B}(s) = \sum_{m \le y}
\frac{b_m}{m^s}$ where $|b_m| \ll |x_d||x_m|$.
 Thus
 \begin{equation}
 \begin{split}
   &  \mathcal{A}(|LA| |G|)
     \le j(d) \tau_3(d) T^{\epsilon}
    \mathcal{A} (|L(s,\psi)|^2 + |L^{'}(s,\psi)|^2)
    | \mathcal{B}(s)||G(s,\psi)|)  \\
  & \ll j(d) \tau_3(d) T^{\epsilon} \mathcal{A}( |L(s,\psi)|^{4} + |L^{'}(s,\psi)|^4)^{1/2}
     \mathcal{A}( |\mathcal{B}(s,\psi)|^4 ) ^{1/4}
     \mathcal{A}( |G(s,\psi)|^4)^{1/4}
     \label{eq:ALAG}
\end{split}
\end{equation}
It suffices to bound $\mathcal{A}(\phi)$ for a variety of $\phi=\phi(s,\psi)$.
We have  the following bounds:
\begin{equation}
  \mathcal{A}(1) \ ,   \
   \mathcal{A}(|L(s,\psi)|^4)  \ , \
  \mathcal{A}( |L'(s,\psi)|^4)
  \ll z^2 T^{\epsilon} \ .
  \label{eq:Abds}
\end{equation}
The first bound is trivial and the last two are due to an argument
of Montgomery \cite{M}. 
Next we analyze $\mathcal{A}(F^2)$, $\mathcal{A}(G^4)$, and 
$\mathcal{A}(\mathcal{B}^4)$. 
Note that 
\[
   F(s,\psi) = \sum_{k \le u} \frac{a(kd) \psi(k)}{k^{s}} \ , \
   G(s,\psi)^2 = \sum_{k \le v^2} \frac{\beta_k}{k^s} \ , \
   \mathcal{B}(s)^2= \sum_{k \le M^2}
   \gamma_k k^{-s}
\]
where the coefficients satisfy
\[
   a(kd) \ll T^{\epsilon} (\tau*x)(d) (\tau*x)(k)
 \ , \
 |\beta_k| \le \tau(k) \ , \
  |\gamma_{k}| \ll x(d)^2 (|x|*|x|)(k)  \ . 
\]  
The large sieve~(\ref{eq:lsi}) inequality yields
\begin{equation}
\begin{split}
   \mathcal{A}(F(s,\psi)^2)
   & \ll  T^{\epsilon}  (\tau*|x|)(d)^2  \sum_{k \le u}(k+z^2 \Lo)  
   \frac{(\tau*|x|)(k)^2}{k} \\
      & \ll (\tau*|x|)(d)^2 T^{\epsilon}(u+x^2) \ ,  \\
  \mathcal{A}(G(s,\psi)^4) 
  &  \ll \sum_{k \le v^2} (k+ z^2 \Lo) \frac{\tau (k)^2}{k}
  \ll T^{\epsilon} (v^2 +z^2)  \ ,  \\
    \mathcal{A}(\mathcal{B}(s)^4)
    & \ll  x(d)^4 \sum_{k \le M^2} (k+z^2 \Lo) \frac{(|x|*|x|)^2(k)}{k}
    \le  T^{\epsilon} x(d)^4 (M^2+z^2)  \ .   
\end{split}
\end{equation}
By~(\ref{eq:Ix}),~(\ref{eq:ALAG}),~(\ref{eq:Abds}) and the bound $(a+b)^{1/n} \ll a^{1/n}+b^{1/n}$ for $a,b >0$, $n \in \mathbb{N}$ we have
\begin{equation}
\begin{split}
  \mathcal{I}(x) & \ll (\tau*|x|)(d)z(u^{1/2}+z) T^{\epsilon} \\
  & + (\tau*|x|)(d) z^{1/2}(u^{1/2}+z) (v^{1/2}+z^{1/2}) T^{\epsilon} \\
  & + j(d) \tau_3(d)| x(d)| z
  (v^{1/2}+z^{1/2})(M^{1/2}+z^{1/2}) T^{\epsilon} \ . 
  \nonumber
\end{split}
\end{equation}
Recalling that $u=z^2$ and $v=T^{1/4}$ this simplifies to 
\[
   \mathcal{I}(z) \ll
   j(d)\tau_3(d)|x(d)| T^{\epsilon}
   \left(
   z^2 +z^{3/2}T^{1/4}+z^{3/2}M^{1/2}+ zT^{1/4}M^{1/2} 
   \right) \  . 
\]
Since $M=T^{\theta}$
\begin{equation}
\begin{split}
   \int_{\eta}^{\frac{M}{k}} z^{-1} d\mathcal{I}(z)
   & \ll 
   j(d)\tau_3(d)|x(d)| T^{\epsilon}
   \left(
   \mbox{$\frac{M}{k}$} + ( \mbox{$\frac{M}{k}$})^{1/2}T^{1/4}
   + (\mbox{$\frac{M}{k}$})^{1/2}M^{1/2}+ T^{1/4}M^{1/2}  
   \right)  \\
   & \ll  j(d) \tau_3(d)|x(d)|
   \left(
   T^{\theta+\epsilon} k^{-1/2} + T^{\frac{\theta}{2}+\frac{1}{4}+\epsilon}
   \right) \ . 
   \nonumber
\end{split}
\end{equation}

\noindent Nathan Ng\\
Department of Mathematics and Statistics \\
University of Ottawa \\
585 King Edward Ave. \\
Ottawa, ON \\
Canada K1N6N5 \\
\email{nng@uottawa.ca}

\end{document}